\documentclass[reqno]{amsart}

\usepackage[utf8]{inputenx} 
\usepackage[T1]{fontenc}    

\usepackage{amssymb}   
\usepackage{amsthm}    
\usepackage{thmtools}  
\usepackage{mathtools} 
\usepackage{mathrsfs}  
\usepackage{bbm}
\usepackage{commath}

\RequirePackage[backend    = biber,
                sortcites  = true,
                giveninits = true,
                doi        = false,
                isbn       = false,
                url        = false,
                style      = numeric-comp,
                maxnames=50]{biblatex}
\DeclareNameAlias{sortname}{family-given}
\DeclareNameAlias{default}{family-given}
\renewbibmacro{in:}{\ifentrytype{article}{}{\printtext{\bibstring{in}\intitlepunct}}}

\usepackage{varioref}
\usepackage{hyperref}
\usepackage[nameinlink, capitalize, noabbrev]{cleveref}

\declaretheorem[style = plain, numberwithin = section]{theorem}

\declaretheorem[style = plain,      sibling = theorem]{lemma}
\declaretheorem[style = plain,      sibling = theorem]{proposition}

\declaretheorem[style = definition, sibling = theorem]{example}
\declaretheorem[style = remark,    numbered = no]{remark}

\DeclareMathOperator{\vol}{vol}

\DeclareMathOperator{\Aut}{Aut}

\DeclareMathOperator{\id}{id}
\newcommand{\B}{\mathcal{B}}
\DeclareMathOperator{\modular}{mod}
\DeclareMathOperator{\supp}{supp}
\newcommand{\bra}{\mathcal{N}}

\newcommand{\N}{\mathbb{N}}   
\newcommand{\Z}{\mathbb{Z}}   
\newcommand{\Q}{\mathbb{Q}}   
\newcommand{\R}{\mathbb{R}}   
\newcommand{\C}{\mathbb{C}}   
\newcommand{\F}{\mathbb{F}}   
\newcommand{\T}{\mathbb{T}}   
\newcommand{\algint}{\mathcal{O}}
\newcommand{\A}{\mathbb{A}}

\def\adeles#1{\A_{#1}}
\def\sadeles#1#2{\A_{#1,#2}}
\def\rationals#1{R_{#1}}
\def\srationals#1#2{R_{#1,#2}}

\def\resprod#1#2#3{\sideset{}{^{#3}}\prod_{#1} #2}
\def\tfp#1{#1 \times \widehat{#1}}

\calclayout

\usepackage{enumitem}
\setlist[enumerate]{label=({\roman*})}

\author[Enstad]{Ulrik Enstad}
\address{Department of Mathematics, University of Oslo, 0851 Oslo, Norway.}
\email{ubenstad@math.uio.no}

\author[Jakobsen]{Mads S.\ Jakobsen}
\address{DESY, Notkestraße 85, 22607 Hamburg, Germany}
\email{mads.jakobsen@desy.de}

\author[Luef]{Franz Luef}
\address{Department of Mathematical Sciences, Norwegian University of Science and Technology, 7491 Trondheim, Norway.}
\email{franz.luef@ntnu.no}

\author[Omland]{Tron Omland}
\address{Norwegian National Security Authority (NSM)}
\email{tron.omland@gmail.com}

\title[Deformations of Gabor frames on the adeles]{Deformations and Balian--Low theorems for Gabor frames on the adeles}


\begin{document}
\maketitle

\begin{abstract}
    \noindent
    We generalize Feichtinger and Kaiblinger's theorem on linear deformations of uniform Gabor frames to the setting of a locally compact abelian group $G$. More precisely, we show that Gabor frames over lattices in the time-frequency plane of $G$ with windows in the Feichtinger algebra are stable under small deformations of the lattice by an automorphism of $\tfp{G}$. The topology we use on the automorphisms is the Braconnier topology. We characterize the groups in which the Balian--Low theorem for the Feichtinger algebra holds as exactly the groups with noncompact identity component. This generalizes a theorem of Kaniuth and Kutyniok on the zeros of the Zak transform on locally compact abelian groups. We apply our results to a class of number-theoretic groups, including the adele group associated to a global field.
\end{abstract}

\section{Introduction}

Bases and frames are indispensable tools in functional analysis, and their stability under small perturbations or deformations is a well-studied topic. The study of perturbations of bases began with the work of Paley and Wiener on nonharmonic Fourier series \cite{Yo80}. A frame is a generalized basis that allows for robust, but non-unique expansions. By definition, a sequence $(g_j)_{j \in J}$ in a Hilbert space $H$ is a \emph{frame} if there exist $A,B > 0$ such that
\begin{equation}
     A \|f \|^2 \leq \sum_{j \in J} |\langle f, g_j \rangle |^2 \leq B \| f \|^2 \label{eq:frame_property}
\end{equation}
for all $f \in H$. Frames were introduced by Duffin and Schaeffer in \cite{DuSc52}, and have found major applications to sampling theory \cite{AlGr01}, wavelet theory \cite{DaGrMe86,Da90} and pseudodifferential operators \cite{GrHe99}, to name a few. A systematic study of perturbations of frames was initiated by Christensen in \cite{Ch95,Ch95-2}, and since then, a number of pertubation results have been proved for various types of frames, including frames of exponentials, wavelet frames and Gabor frames \cite{Ba97,SuZh99-1,SuZh02,CaCh08}.

There are natural notions of perturbations and deformations of structured function systems like Gabor systems and wavelet systems. A Gabor system is a set $\mathcal{G}(g,\Delta) \coloneqq \{ e^{2\pi i \langle \omega, \cdot \rangle } g(\cdot-x) : (x,\omega) \in \Delta \}$ of time-frequency translates of a single function $g \in L^2(\R^n)$, where the translates come from a discrete point set $\Delta$ in the time-frequency plane $\R^{2n}$. A \emph{Gabor frame} is a Gabor system that satisfies the frame property \eqref{eq:frame_property}. One can then raise the natural question of whether the frame property of a Gabor system is retained after a sufficiently small deformation of either the point set $\Delta$ or the window $g$.

The first result in this direction is due to Feichtinger and Kaiblinger \cite{FeKa04}, and concerns \emph{linear} deformations of \emph{uniform} Gabor frames with windows in the modulation space $M^1(\R^n)$, also known as the Feichtinger algebra $S_0(\R^n)$ \cite{Fe81-2}. A linear deformation of the point set $\Delta$ is implemented by a matrix $A \in GL_{2n}(\R)$, while uniform means that the point set $\Delta$ has the structure of a lattice (a discrete, cocompact subgroup) in $\R^{2n}$. Their main result states that Gabor frames of this type are stable both under linear deformations of the point set $\Delta$ and under small changes in the $S_0(\R^n)$-norm of $g$:

\begin{theorem}[Feichtinger--Kaiblinger \cite{FeKa04}]\label{thm:feka}
Let $g \in S_0(\R^n)$, and let $\Delta$ be a lattice in $\R^{2n}$. If $\mathcal{G}(g,\Delta)$ is a Gabor frame, then there exists a neighbourhood $W$ of $I_{2n} \in GL_{2n}(\R)$ and a neighbourhood $U$ of $g \in S_0(\R^n)$ such that $\mathcal{G}(h,A\Delta)$ is a Gabor frame for all $A \in W$ and $h \in U$.
\end{theorem}

Ascenci, Feichtinger and Kaiblinger generalized the result to linear deformations of arbitrary point sets using the theory of pseudodifferential operators in \cite{AsFeKa14}. Later, stability of nonuniform Gabor frames under a large class of nonlinear deformations was proved by Gröchenig, Ortega-Cerdà and Romero  \cite{GrOrRo15}.

In this paper, we initiate the study of deformations of Gabor frames on a (second-countable) locally compact abelian group $G$. In this setting, a time-frequency shift of $f \in L^2(G)$ by $(x,\omega) \in \tfp{G}$ is defined as follows:
\begin{align*}
    \pi(x,\omega)f(t) = \omega(t) f(x^{-1}t),  \;\;\;\;\;  t \in G.
\end{align*}
Here, $\widehat{G}$ denotes the Pontryagin dual of $G$. The Gabor system with window $g \in L^2(G)$ over the discrete point set $\Delta \subseteq \tfp{G}$ is defined as
\[ \mathcal{G}(g,\Delta) = \{ \pi(z) g : z \in \Delta \} \subseteq L^2(G) .\]
Even though time-frequency analysis is most commonly done in the Euclidean case $G = \R^n$, many of the fundamental results hold in the locally compact abelian setting, such as density and duality results for uniform Gabor frames and Zak transform considerations \cite{Gr98,JaLe16,JaLe16-2}. Gabor analysis on other groups appear naturally in applications. For instance, the integers $\Z$, finite cyclic groups $\Z/n\Z$ and the circle group $\T$ are important when sampling or periodizing Gabor frames on $\R$ \cite{Ja97,So07,Ka05,JaLu19}. Other considerations on general locally compact abelian groups include \cite{KaKu98,GrSt07,JaLu19}, and Balian--Low type phenomena were proved to hold for finite groups in \cite{NiOl19}.

In general, one refers to the groups $\R$, $\Z$, $\T$ and $\Z/n\Z$ and their products as \emph{elementary} LCA groups. A rather different class of LCA groups come from algebraic number theory. Given an algebraic number field or a global function field, the associated adele ring is a restricted product of completions of the given field with respect to its different absolute values. Among the many applications of the adele ring is an elegant statement of the Artin reciprocity law from class field theory \cite{Ta67}. In the case of the rational numbers, the associated adele ring is built from $\R$ and the fields of $p$-adic numbers $\Q_p$ for a prime $p$. Gabor frames on the adele ring of the rational numbers were recently constructed in \cite{EnJaLu19}.

In the Euclidean case $G = \R^n$, there is an immediate notion of a linear deformation of $G$, and the group $GL_n(\R)$ of such deformations is again Euclidean, thus providing a familiar topology to formulate deformation results. In the general context of a locally compact abelian group $G$, there is a natural topology on the automorphism group $\Aut(G)$, called the \emph{Braconnier topology}. First introduced by Braconnier in \cite{Br48}, it is a topology defined in such a way that the operations of composition and taking inverses of automorphisms are continuous. Thus, $\Aut(G)$ itself becomes a topological group. In specific situations, the Braconnier topology coincides with the topology one would expect---in particular, $\Aut(\R^n) \cong GL_n(\R)$ as topological groups. Our first result is a generalization of the result of Feichtinger and Kaiblinger (\Cref{thm:feka}), and states that in the setting of a locally compact abelian group $G$, we can replace $GL_{2n}(\R) = \Aut(\tfp{\R^n})$ with the group $\Aut(\tfp{G})$, equipped with the Braconnier topology:

\begin{theorem}[cf.\ \Cref{thm:pert}]\label{intro:thm1}
Let $G$ be a locally compact abelian group, let $\Delta$ be a lattice in $\tfp{G}$ and let $g \in S_0(G)$. If $\mathcal{G}(g,\Delta)$ is a Gabor frame, then there exist a neighbourhood $W$ of $\id_G \in \Aut(\tfp{G})$ and a neighbourhood $U$ of $g \in S_0(G)$ such that $\mathcal{G}(h,\alpha(\Delta))$ is a Gabor frame for all $\alpha \in W$ and $h \in U$.
\end{theorem}

Our main tools are the duality theory for uniform Gabor frames on LCA groups \cite{JaLe16}, continuity of various maps between the Feichtinger algebra of different groups \cite{Ja18}, and the existence of dual windows in the Feichtinger algebra \cite{GrLe04,Au21}. At the heart of our approach is the continuity of dilation on $S_0(G)$ by automorphisms from $\Aut(G)$, see \Cref{thm:strong_cont_s0}. We establish this result using a particular description of the Feichtinger algebra, see \Cref{prop:feichtinger_description}, and a similar continuity result for $L^1(G)$ already proved in \cite{Br48}, see \Cref{prop:strong_cont}.

To give examples of applications of \Cref{intro:thm1} to groups other than $G = \R^n$, we introduce a class of groups which we call \emph{higher dimensional $S$-adeles}. These groups are built from a global field $K$, a set $S$ of finite places on $K$ and a natural number $n$. If $S$ is the set of all finite places on $K$ and $n=1$, one recovers the usual adele ring of $K$, and if $K = \Q$ and $S = \emptyset$, one recovers $G = \R^n$. We apply \Cref{intro:thm1} to all of these groups when $K = \Q$, and the following is the statement for the higher dimensional full adele ring over the rationals:

\begin{theorem}[cf.\ \Cref{prop:lattices}, \Cref{thm:pert_adeles}]
Let $\A_{\Q}$ denote the rational adele ring, and identify $\Q$ with its diagonal copy inside $\A_{\Q}$ (see \Cref{sec:adeles}). Then any lattice $\Delta$ in $\tfp{\A_{\Q}^n} \cong \A_{\Q}^{2n}$ is of the form
\[ \Delta = A \Q^{2n} \coloneqq \{ (A_{\infty}q, A_2q, A_3q, A_5q, \ldots ) : q \in \Q^{2n} \} \]
for a collection of matrices $A = (A_{\infty}, A_2, A_3, A_5, \ldots )$ where $A_{\infty} \in GL_{2n}(\R)$, $A_p \in GL_{2n}(\Q_p)$, with $A_p \in GL_{2n}(\Z_p)$ for all but finitely many primes $p$. Moreover, if $\mathcal{G}(g,\Delta)$ is a Gabor frame with $\Delta$ as above, then there exist open neighbourhoods $W_{\infty} \subseteq GL_{2n}(\R)$ of $A_{\infty}$, $W_p \subseteq GL_{2n}(\Q_p)$ of $A_p$, $p$ prime, with $W_p = GL_{2n}(\Z_p)$ for all but finitely many $p$, and a neighbourhood $U$ of $g$ in $S_0(\A_{\Q})$ such that $\mathcal{G}(h,B\Q^{2n})$ is a Gabor frame whenever $B = (B_{\infty}, B_2, B_3, B_5, \ldots) \in W_{\infty} \times W_2 \times W_3 \times \cdots$ and $h \in U$.
\end{theorem}

\subsection*{The Balian--Low theorem}

Our second main goal is to link the deformation result in \Cref{intro:thm1} to The Balian--Low theorem for locally compact abelian groups. In the Euclidean setting, the Balian--Low theorem is a cornerstone of time-frequency analysis, and concerns the nonexistence of well-localized Gabor frames at the critical density. One of the consequences of the deformation result of Feichtinger--Kaiblinger (\Cref{thm:feka}) is that it gives as an immediate consequence the Balian--Low theorem for all lattices of volume 1 in the time-frequency plane. We present the theorem and the proof given in \cite{AsFeKa14} below.

\begin{theorem}[Feichtinger--Kaiblinger \cite{FeKa04}]\label{thm:bltr}
Let $\Delta$ be a lattice in $\R^{2n}$ of volume 1, and let $g \in S_0(\R^n)$. Then the Gabor system $\mathcal{G}(g,\Delta)$ is not a frame for $L^2(\R^n)$.
\end{theorem}

\begin{proof}
Let $\Delta = A \Z^{2n}$ be a lattice in $\R^{2n}$ with $\vol(\Delta) = |\det A| = 1$. Let $\mathcal{G}(g,\Delta)$ be a Gabor frame with $g \in S_0(\R^n)$. Then by \Cref{thm:feka}, there exists an $\epsilon > 0$ such that $\mathcal{G}(g,\Delta')$ is a Gabor frame, where $\Delta' = ((1+\epsilon)A)\Z^{2n}$. But then
\[ \vol(\Delta') = |\det((1+\epsilon)A)| = (1+\epsilon)^n > 1 \]
which contradicts the density theorem for Gabor frames.
\end{proof}

For the above proof to work, it is crucial that the determinant function is open, i.e.\ it maps open sets to open sets. The (absolute value of the) determinant describes how the Lebesgue measure of a set changes as a linear deformation is applied. In the setting of a locally compact abelian group $G$, the role of the determinant is played by the \emph{Braconnier modular function} $\modular_G \colon \Aut(G) \to (0,\infty)$ which is defined via the property
\begin{equation}
    \mu(\alpha(S)) = \modular_G(\alpha) \mu(S)
\end{equation}
where $\mu$ is any Haar measure on $G$ and $S$ is any Borel set of positive measure. The Braconnier modular function is continuous with respect to the Braconnier topology on $\Aut(G)$. It is clear from the proof of \Cref{thm:bltr} that one gets a similar result in the locally compact abelian setting, provided that the Braconnier modular function of $\tfp{G}$ is open. In \Cref{thm:idcomp_equivalences}, which is one of our main technical results, we prove that the openness of $\modular_G$ is equivalent to a number of other conditions. One of them is the noncompactness of the identity component of $G$, as well as the openness of the Braconnier modular function of $\tfp{G}$. The characterization relies on the structure theory of locally compact abelian groups and van Dantzig's theorem for totally disconnected groups. Using this characterization, we proceed, in the exact same manner as in the proof of \Cref{thm:bltr}, to show the following:

\begin{theorem}[cf.\ \Cref{thm:balian_low}]\label{intro:thm3}
Let $G$ be a locally compact abelian group with noncompact identity component. Let $\Delta$ be a lattice in $\tfp{G}$ of volume $1$, and let $g \in S_0(G)$. Then the Gabor system $\mathcal{G}(g,\Delta)$ cannot be a frame for $L^2(G)$.
\end{theorem}

On the other hand, if $G$ has compact identity component, then $\modular_G$ takes values in the rational numbers (\Cref{thm:idcomp_equivalences}), so the argument for \Cref{intro:thm3} cannot be carried out.

It was already observed by Kaniuth and Kutyniok in \cite{KaKu98} that Balian--Low phenomena are valid in certain locally compact abelian groups with noncompact identity component. Their main result concerns the zeros of the Zak transform, which is an essential tool used to study Gabor systems over lattices in the time-frequency plane of the form $\Lambda \times \Lambda^{\perp}$ for $\Lambda$ a lattice in $G$. The result goes as follows:

\begin{theorem}[Kaniuth--Kutyniok \cite{KaKu98}]
Let $G$ be a locally compact abelian group that is compactly generated and has noncompact identity component. Then for every lattice $\Lambda$ in $G$ and every $f \in L^2(G)$ such that the Zak transform $Z_{\Lambda} f$ is continuous, $Z_{\Lambda} f$ has a zero.
\end{theorem}

A vital assumption in their argument is that the group is compactly generated. Using \Cref{intro:thm3}, we go from a Balian--Low theorem to a statement about the zeros of the Zak transform, and we are able to remove the assumption of $G$ being compactly generated:

\begin{theorem}[cf.\ \Cref{thm:zak_zeros}]
Let $G$ be a locally compact abelian group with noncompact identity component. Then for any lattice $\Lambda$ in $G$ and any $f \in L^2(G)$ for which the Zak transform $Z_{\Lambda} f$ is continuous, $Z_{\Lambda} f$ has a zero.
\end{theorem}

Note that in contrast to the main result of \cite{KaKu98}, the above result applies to e.g.\ the group of adeles, as this is not a compactly generated group. Note also that while we restrict to second-countable groups in the present paper, the result of \cite{KaKu98} holds without this assumption.

Having established a general Balian--Low theorem for $S_0(G)$ in groups with noncompact identity component, the question of what happens when the identity component is compact arises---or equivalently, when $G$ contains a compact open subgroup (see \Cref{thm:idcomp_equivalences}). Gröchenig observed in \cite{Gr98} that if $G$ contains a compact, open subgroup, then the Balian--Low theorem fails in the following strong sense: There exists a discrete set $\Delta$ in $\tfp{G}$ and a function $g \in S_0(G)$ such that $\mathcal{G}(g,\Delta)$ is an orthonormal basis for $L^2(G)$. However, the question still stands whether one can obtain a result for uniform Gabor frames, i.e.\ if $\Delta$ can be chosen to be a lattice (in the case that lattices exist). In this paper, we prove the following:

\begin{theorem}[cf.\ \Cref{thm:blt_failure}]\label{intro:thm5}
Let $G$ be a locally compact abelian group with compact identity component, and let $\Lambda$ be a lattice in $G$. Then there exists a $g \in S_0(G)$ such that $\mathcal{G}(g,\Lambda \times \Lambda^{\perp})$ is an orthonormal basis for $L^2(G)$.
\end{theorem}

Hence, the Balian--Low theorem for uniform lattices fails in a very strong sense if the group has compact identity component. To prove \Cref{intro:thm5}, we rely on our characterization of groups with compact identity component in \Cref{thm:idcomp_equivalences}. Of course, for \Cref{intro:thm5} to be applied, one needs the existence of a lattice in the first place, and there are many LCA groups without lattices, the $p$-adic numbers $\Q_p$ being an example.

Finally, we apply our characterization of the groups $G$ for which the Balian--Low theorem for $S_0(G)$ holds to the higher dimensional $S$-adeles associated to a global field $K$. We show that we do get a Balian--Low theorem precisely when $K$ is an algebraic number field. This generalizes \cite[Theorem C]{En19}, where it was shown that a Balian--Low theorem for a certain lattice in the time-frequency plane of the group $\R \times \Q_p$ holds. Below, we state our result for the full adeles associated to a global field:

\begin{theorem}[cf.\ \Cref{thm:blt_adeles1}, \Cref{thm:blt_adeles2}]
Let $K$ be a global field and let $n$ be a natural number. Denote by $G = \adeles{K}^n$ the $n$-dimensional adeles associated to $K$. Then the following hold:
\begin{enumerate}
    \item If $K$ is an algebraic number field, then the Balian--Low theorem holds for $G$. That is, for any lattice $\Delta$ in $\tfp{G}$ of volume $1$ and any $g \in S_0(G)$, the Gabor system $\mathcal{G}(g,\Delta)$ is not a frame for $L^2(G)$.
    \item If $K$ is a global function field, then the Balian--Low theorem does not hold for $G$. In fact, for every lattice $\Lambda$ in $G$, there exist $g \in S_0(G)$ such that $\mathcal{G}(g,\Lambda \times \Lambda^{\perp})$ is an orthonormal basis for $L^2(G)$.
\end{enumerate}
\end{theorem}

Finally, we want to remark that our deformation and Balian--Low type results only cover the case of uniform sampling, i.e.\ the point set $\Delta$ is a lattice. For Gabor frames on $G = \R^n$, the state of the art is non-uniform Gabor frames, such as in \cite{AsFeKa14,GrOrRo15}. In fact, some of the cornerstones of Gabor analysis on $\R^n$, such as the density theorems and the Balian--Low theorem, have in the last few years found non-uniform generalizations beyond the setting of time-frequency analysis \cite{GrRo18,GrHaKl17}. In a recent paper of Gröchenig, Romero, Rottensteiner and van Velthoven \cite{GrRoRo19}, a (nonuniform) Balian--Low type theorem for homogeneous Lie groups was established. While their main theorem generalizes the Balian--Low theorem for Gabor frames on $G=\R^n$, it does not generalize \Cref{intro:thm3}: Many locally compact abelian groups, e.g.\ the adeles over the rationals, are not Lie groups. Thus, an interesting question is whether there exists a Balian--Low type theorem that simultaneously generalizes the main result of \cite{GrRoRo19} and \Cref{intro:thm3}.

The paper is structured as follows: In \Cref{sec:groups_automorphisms}, we cover the basics of locally compact abelian groups and their automorphisms, including \Cref{thm:idcomp_equivalences} where we characterize groups with noncompact identity component. In \Cref{sec:tfa}, we review time-frequency analysis on LCA groups. In \Cref{sec:deformation}, we prove continuity and deformation results for Gabor frames on locally compact abelian groups, including the main result \Cref{thm:pert}. In \Cref{sec:blt}, we prove that the Balian--Low theorem for $S_0(G)$ holds for groups with noncompact identity component and fails for groups with compact identity component. In \Cref{sec:adeles}, we introduce the higher-dimensional $S$-adeles and apply our results for LCA groups to them.

\subsection*{Acknowledgements}

The first and third author wish to thank Hans Feichtinger for enlightening discussions in Vienna and Lisbon. The first author wants to thank Nadia Larsen and Sven Raum for helpful discussions about the Braconnier topology and totally disconnected groups, respectively.

\section{Locally compact abelian groups and their automorphisms}\label{sec:groups_automorphisms}

Throughout the paper, we will assume, unless otherwise stated, that $G$ is a locally compact (Hausdorff) abelian group. Such a group always carries a translation invariant regular Borel measure called a Haar measure, which is unique up to a positive constant. We will also add the standard assumption that $G$ is second-countable in order to use results from \cite{JaLe16,JaLe16-2,Ja18}, although we remark that many stated results throughout the present paper hold without this assumption. We write the group operation multiplicatively and we denote by $1$ the identity element of $G$.

If $H$ is a closed subgroup of $G$, then both $H$ and $G/H$ are locally compact abelian groups. The relation between the Haar measure on these three groups can be given as follows: Once two out of three Haar measures have been chosen on $G$, $H$ and $G/H$, the last one can be chosen so that Weil's formula \cite[Proposition 3.3.11]{Re20} holds:
\begin{align}
    \int_G f(x) \dif{\mu_G(x)} = \int_{G/H} \int_H f(xy) \dif{\mu_H(y)} \dif{\mu_{G/H}(xH)} \ \ \text{for all} \ \ f \in C_c(G) . \label{eq:weil}
\end{align}
A closed subgroup $H$ of $G$ is called \emph{cocompact} if the quotient group $G/H$ is compact. A subgroup $\Lambda$ is a \emph{lattice} in $G$ if it is both discrete and cocompact. If one fixes a Haar measure $\mu_G$ on $G$ and chooses the counting measure on $\Lambda$, then there exists a unique measure $\mu_{G/\Lambda}$ on $G/\Lambda$ such that Weil's formula \eqref{eq:weil} is satisfied. Since $G/\Lambda$ is compact, the measure $\mu_{G/\Lambda}$ is finite, and we define the \emph{volume} of $\Lambda$ to be the number
\begin{equation}
    \vol(\Lambda) = \mu_{G/\Lambda}(G/\Lambda) . \label{eq:volume}
\end{equation}
Note that $\vol(\Lambda)$ depends on $\mu_G$.

Denote by $\widehat{G}$ the Pontryagin dual of $G$. If $H$ is a closed subgroup of $G$, the set
\begin{equation}
    H^{\perp} \coloneqq \big\{ \omega \in \widehat{G} : \text{$\omega(x) = 1$ for all $x\in H$} \big\}
\end{equation}
is a closed subgroup of $\widehat{G}$ called the \emph{annihilator} of $\Lambda$. If $\Lambda$ is a lattice in $G$, then $\Lambda^{\perp}$ is a lattice in $\widehat{G}$ \cite[Lemma 3.1]{Ri88}. A proof of the following formula is found in \cite{Gr98}:
\begin{equation}
    \vol(\Lambda)\vol(\Lambda^{\perp}) = 1 . \label{eq:vol_lattice_perp}
\end{equation}

We will need the following lemma later, which contains elementary observations on lattices in LCA groups.

\begin{lemma}\label{lem:lattice-constructions}
The following hold:
\begin{itemize}
\item[(i)]
Let $\Lambda\subseteq H\subseteq G$ be a sequence of closed subgroups such that $G/H$ is compact.
Then $\Lambda$ is a lattice in $H$ if and only if $\Lambda$ is a lattice in $G$.
\item[(ii)]
Let $\varphi\colon G\to H$ be an open surjective map such that $\ker\varphi$ is compact.
Let $\Lambda$ be a lattice in $G$.
Then $\varphi(\Lambda)$ is a lattice in $H$.
In particular, this applies when $K$ is a compact subgroup of $G$ and $\varphi$ is the quotient map $G\to G/K$.
\item[(iii)]
Let $L\subseteq\Lambda\subseteq G$ be a sequence of closed subgroups such that $\Lambda/L$ is finite.
Then $\Lambda$ is a lattice in $G$ if and only if $L$ is a lattice in $G$.
\end{itemize}
\end{lemma}

\begin{proof}
(i) Note first that $H/\Lambda$ is closed in $G/\Lambda$ (\cite[5.39b]{HeRo63}). Thus if $G/\Lambda$ is compact, then $H/\Lambda$ is compact. Conversely, if $H/\Lambda$ is compact, then by \cite[5.25, 5.35]{HeRo63} it follows that $G/\Lambda$ is compact.

If $\Lambda$ is discrete in $H$, then there exists an open set $U$ of $G$ such that $(U\cap H)\cap\Lambda=\{1\}$, but then $U\cap\Lambda=\{1\}$, so $\Lambda$ is discrete in $G$. The converse is trivial.

(ii)
Since $\Lambda$ is discrete in $G$, we can find an open set $U$ around the kernel $K$ such that $\Lambda\cap K\subseteq U$ and $U\cap(\Lambda\cap K^c)=\varnothing$ (i.e., $U\cap\Lambda=K\cap\Lambda$). By \cite[Theorem~4.10]{HeRo63} there exists an open set $V$ around $1$ such that $VK\subseteq U$. Therefore $\pi(VK)\cap\pi(\Lambda)\subseteq\pi(K\cap\Lambda)=\{1\}$, so $1$ is isolated in $G/K$.

Moreover, $\Lambda K$ is closed in $G$ since it is the inverse image of $\pi(\Lambda)$). Thus, there is a surjective map $G/\Lambda\to G/(\Lambda K)$, given by $x\Lambda\mapsto x\Lambda K$. This means that the image is compact.

It is easy to see that the converse does not hold.

(iii)
Clearly, $\Lambda$ is discrete in $G$ if and only if $L$ is discrete in $G$. Again by \cite[5.25, 5.35]{HeRo63} it follows that $G/L$ is compact if and only if $G/\Lambda$ is compact. Note that there is a quotient map from $G/L$ onto $G/\Lambda$.
\end{proof}

The \emph{identity component} $G_0$ of $G$ is the connected component of $G$ containing the identity element $1$ of $G$. This is a closed, connected subgroup of $G$. We call $G$ \emph{totally disconnected} if its underlying topology is disconnected, i.e.\ the connected components of $G$ are exactly the one-point sets. Equivalently, the identity component $G_0$ is the trivial subgroup of $G$. The quotient $G/G_0$ is always a totally disconnected, locally compact abelian group, see \cite[Theorem 7.3]{HeRo63}.

We will need the following famous result on totally disconnected groups, known as van Dantzig's theorem:

\begin{proposition}\label{prop:vandantzig}
Let $G$ be a locally compact abelian, totally disconnected group. Then every neighbourhood of the identity of $G$ contains a compact, open subgroup.
\end{proposition}

See \cite[Theorem 7.7]{HeRo63} for a proof. A consequence of van Dantzig's theorem is the following description of the identity component (see \cite[Theorem 7.8]{HeRo63} for a proof):
\begin{equation}
G_0 = \bigcap \big\{ H : \text{$H$ is an open subgroup of $G$}\, \big\} . \label{eq:identity_component_intersection}
\end{equation}

\subsection{Group automorphisms}
By an automorphism of $G$, we will mean a group isomorphism $G \to G$ which, at the same time, is a homeomorphism with respect to the topology of $G$. The set $\Aut(G)$ of automorphisms of $G$ becomes a group with respect to composition of automorphisms. This group carries a natural topology which makes it into a topological group itself. The topology was introduced by Braconnier in \cite{Br48}. The following proposition describes a neighbourhood basis at the identity for this topology, see \cite[Theorem 26.5]{HeRo63}.

\begin{proposition}\label{prop:braconnier_def}
Let $G$ be a topological group. Given a compact set $K \subseteq G$ and an (open) neighbourhood $U$ of the identity of $1$, define
\[ \bra(K,U) = \{ \alpha \in \Aut(G) : \text{$\alpha(x)x^{-1} \in U$ and $\alpha^{-1}(x)x^{-1} \in U$ for all $x \in K$} \} . \]
Then these sets form a neighbourhood basis at the identity for $\Aut(G)$, and we call the topology they generate the \emph{Braconnier topology} on $\Aut(G)$.
\end{proposition}

For the rest of the paper, we will view $\Aut(G)$ as a topological group with the Braconnier topology. One thing to note is that even if $G$ is locally compact (Hausdorff), $\Aut(G)$ need not be, see \cite[12.1.3]{Pa01}.

Fix a Haar measure $\mu$ on $G$ and let $\alpha \in \Aut(G)$. Then the mapping $S \mapsto \mu(\alpha(S))$ for Borel sets $S \subseteq G$ defines another Haar measure on $G$, so there exists a constant $\modular_G(\alpha) \in (0,\infty)$ such that
\begin{equation}
    \mu(\alpha(S)) = \modular_G(\alpha) \, \mu(S)
\end{equation}
for all Borel sets $S \subseteq G$. The constant $\modular_G(\alpha)$ is independent of the choice of Haar measure. Moreover, it defines a continuous group homomorphism
\begin{equation}
\modular_G \colon \Aut(G) \to (0,\infty) \label{eq:modular_function}
\end{equation}
where the latter is the group of positive real numbers under multiplication. This homomorphism is called the \emph{Braconnier modular function}. Note that we use the original convention due to Braconnier as in \cite[p.\ 75]{Br48}, rather than the convention of e.g.\ \cite[p.\ 1275]{Pa01}. Extending to integrals, one obtains
\begin{equation}
    \int_G f(\alpha(x)) \dif{x} = \modular_G(\alpha)^{-1} \int_G f(x) \dif{x} \label{eq:modular_integral}
\end{equation}
for $f \in L^1(G)$.

\begin{example}
If $G = \R$, then any automorphism $\alpha \in \Aut(\R)$ is given by $\alpha(x) = ax$ for some $a \in \R^{\times} = \R \setminus \{ 0 \}$. This gives an isomorphism $\Aut(\R) \cong \R^{\times}$, and under this isomorphism, the Braconnier modular function is given by $\modular_{\R}(a) = |a|$. Similarly, if $G = \Q_p$, then $\Aut(\Q_p) \cong \Q_p^{\times}$ and $\modular_{\Q_p}(a) = |a|_p$, where $|\cdot |_p$ denotes the $p$-adic absolute value.
\end{example}

\begin{example}
(\cite[Corollary 3]{We74}). Let $F$ be a locally compact field, and set $G = F^n$ for some $n \in \N$. Then we can identify $\Aut(G)$ with $\textnormal{GL}_n(F)$, and the Braconnier modular function on $G$ is given by
\[ \modular_G(A) = \modular_F(\det A) .\]
\end{example}

If $\alpha \in \Aut(G)$ and $H$ is a closed subgroup of $G$, we say that $H$ is invariant under $\alpha$ if $\alpha(H) \subseteq H$. If $H$ is invariant under $\alpha$ for all $\alpha \in \Aut(G)$, then $H$ is called a \emph{characteristic subgroup} of $G$. Note that if $H$ is a characteristic subgroup and $\alpha \in \Aut(G)$, then in fact $\alpha(H) = H$, which comes from the fact that $H$ is also invariant under $\alpha^{-1}$.

If $H$ is a closed subgroup invariant under $\alpha$, then we get induced automorphisms $\alpha|_H \in \Aut(H)$ and $\tilde{\alpha} \in \Aut(G/H)$ given by $\tilde{\alpha}(xH) = \alpha(x)\,H$ for $x \in G$. By replacing $f$ in Weil's formula \eqref{eq:weil} with $f \circ \alpha$ and using \eqref{eq:modular_integral}, one deduces that
\begin{equation}
    \modular_G(\alpha) = \modular_{H}(\alpha|_H) \,  \modular_{G/H}(\tilde{\alpha}) . \label{eq:modular_prod}
\end{equation}
In particular, if $H$ is a characteristic subgroup of $G$, then \eqref{eq:modular_prod} holds for all $\alpha \in \Aut(G)$.

The identity component $G_0$ is always a characteristic subgroup of $G$, since an automorphism $\alpha$ must map the connected component of $x \in G$ into the connected component of $\alpha(x)$, and the identity of $G$ is mapped to itself.

Since we assume that $G$ is abelian, we can also consider the automorphism group of the dual group of $G$. The relation to the automorphism group of the original group is given as follows, see \cite[Theorem 26.9]{HeRo63}:

\begin{proposition}\label{prop:dual_iso}
Let $G$ be a locally compact abelian group. Then the map $\Aut(G) \to \Aut(\widehat{G})$ given by $\alpha \mapsto \widehat{\alpha}$ where
\[ \widehat{\alpha}(\omega) = \omega \circ \alpha \]
for $\omega \in \widehat{G}$, is an anti-isomorphism of topological groups.
\end{proposition}

It will be important to us to determine when the Braconnier modular function of a group is open. The following lemma states that when a group factorizes into a product containing a factor of $\R$, then this is indeed the case:

\begin{lemma}\label{lem:R_open}
Let $G$ be a locally compact abelian group. Then the Braconnier modular function of $\R \times G$ is open.
\end{lemma}

\begin{proof}
First, note that $\modular_{\R \times G}$ is surjective. Indeed, if $\alpha \in \Aut(\R)$ and $\beta \in \Aut(G)$, denote by $\alpha \times \beta$ the automorphism of $\R \times G$ given by $(\alpha \times \beta)(x,y) = (\alpha(x),\beta(y))$ for $(x,y) \in \R \times G$. Then for any $t > 0$ we have that $\modular_{\R \times G}(\cdot t \times \id_G) = t$, where $\cdot t \colon \R \to \R$ denotes multiplication by $t$.

Next, let $U$ be an open set in $\R \times G$ containing $1$ and let $K$ be a compact set in $\R \times G$. Consider the element $\bra(K,U)$ of the neighborhood basis at the identity of $\Aut(\R \times G)$ described in \Cref{prop:braconnier_def}. We will show that $\modular_{\R \times G}(\bra(K,U))$ has nonempty interior, from which it will follow by \cite[(5.40) (b)]{HeRo63} that $\modular_{\R \times G}$ is an open map. We can find open sets $U_1$ and $U_2$ in $\R$ and $G$, respectively, containing the identity, such that $U_1\times U_2\subseteq U$. Moreover, there are compact sets $K_1$ and $K_2$ in $\R$ and $G$, respectively, such that $K\subseteq K_1\times K_2$. Set $S = \{ \alpha \times \id_G : \alpha \in \bra(K_1,U_1) \}$. Then $S \subseteq \bra(K,U)$, so $\modular_{\R}(\bra(K_1,U_1)) = \modular_{\R \times G}(S) \subseteq \modular_{\R \times G}(\bra(K,U))$. Since $\modular_{\R}$ is an open map, this shows that $\modular_{\R \times G}(\bra(K,U))$ has nonempty interior, which finishes the proof.
\end{proof}

On the other hand, if $G$ has compact identity component, then its Braconnier modular function is not open. In fact, we have the following:

\begin{proposition}\label{prop:modular_rational}
Suppose $G$ is a locally compact abelian group with compact identity component. Then the Braconnier modular function of $G$ takes values in the rational numbers.
\end{proposition}

\begin{proof}
Suppose that $G_0$, the identity component of $G$, is compact. Then the modular function of $G_0$ is constantly equal to $1$. Since $G_0$ is a characteristic subgroup, we have from \eqref{eq:modular_prod} that
\begin{equation}
\modular_G(\alpha) = \modular_{G/G_0}(\tilde{\alpha}) \label{eq:modular_g0}
\end{equation}
for all $\alpha \in \Aut(G)$. Now $G/G_0$ is totally disconnected, so by \cite[Corollary 12.3.18]{Pa01}, $\modular_{G/G_0}$ takes values in the rational numbers. But then by \eqref{eq:modular_g0}, $\modular_G$ also takes values in the rational numbers. This finishes the proof.
\end{proof}

We will need the following result, which is one of the consequences of the structure theory for locally compact abelian groups \cite[Theorem 24.30]{HeRo63}:

\begin{proposition}\label{prop:structure_result}
Let $G$ be a locally compact abelian group. Then $G \cong \R^d \times H$, where $d \geq 0$ is an integer and $H$ is a locally compact abelian group containing a compact, open subgroup. Furthermore, if $\R^d \times H \cong \R^{d'} \times H'$ where both $H$ and $H'$ are locally compact abelian groups containing compact, open subgroups, then $d = d'$ and $H \cong H'$.
\end{proposition}

We are now ready to give various characterizations of the openness of the Braconnier modular function of $G$.

\begin{theorem}\label{thm:idcomp_equivalences}
Let $G$ be a locally compact abelian group. Then the following are equivalent:
\begin{enumerate}
    \item\label{it:idcomp1} $G \cong \R^d \times H$ where $d \geq 1$ and $H$ is a group containing a compact, open subgroup.
    \item\label{it:idcomp2} $G \cong \R \times H'$ for some locally compact abelian group $H'$.
    \item\label{it:idcomp3} $G$ has noncompact identity component.
    \item\label{it:idcomp4} $G$ has no compact, open subgroups.
    \item\label{it:idcomp5} The Braconnier modular function of $G$ is open.
    \item\label{it:idcomp6} The Braconnier modular function of $\tfp{G}$ is open.
\end{enumerate}
\end{theorem}

\begin{proof}
By \cref{prop:structure_result}, $G \cong \R^d \times H$ where $d\geq 0$ is an integer and $H$ is a locally compact abelian group containing a compact, open subgroup, $d$ is unique and $H$ is unique up to isomorphism. The identity component of $G$ is then given by $\R^d \times H_0$, where $H_0$ denotes the identity component of $H$. The following implications are sufficient to prove the equivalence of $\ref{it:idcomp1}$--$\ref{it:idcomp6}$:

$\ref{it:idcomp1} \iff \ref{it:idcomp2}$: It is clear that $\ref{it:idcomp1}$ implies $\ref{it:idcomp2}$, and the reverse implication comes from using \Cref{prop:structure_result} on $H'$.

$\ref{it:idcomp3} \iff \ref{it:idcomp4}$: Suppose $G$ contains a compact open subgroup $K$. Then by \eqref{eq:identity_component_intersection}, $G_0 \subseteq K$, so $G_0$ must also be compact. Conversely, suppose $G_0$ is compact. Then the quotient map $p \colon G \to G/G_0$ is proper. Since $G/G_0$ is totally disconnected, van Dantzig's theorem (\Cref{prop:vandantzig}) gives the existence of a compact, open subgroup $\tilde{K}$ in $G/G_0$. The preimage $p^{-1}(\tilde{K})$ is then a compact, open subgroup of $G$.

$\ref{it:idcomp1} \implies \ref{it:idcomp3}$: If $d \geq 1$ then $G_0 = \R^d \times H_0$ is not compact.

$\ref{it:idcomp4} \implies \ref{it:idcomp1}$: We prove the contrapositive. If $G$ is not of the form $\R^{d} \times H$ with $H$ containing a compact open subgroup and $d \geq 1$, then by the uniqueness part of \Cref{prop:structure_result}, we must have $d = 0$. But then $G = H$ and $H$ contains a compact open subgroup.

$\ref{it:idcomp2} \implies \ref{it:idcomp5}$: This is \cref{lem:R_open}.

$\ref{it:idcomp5} \implies \ref{it:idcomp4}$: We prove the contrapositive. If $G$ contains a compact, open subgroup, then by \cref{prop:modular_rational} the Braconnier modular function of $G$ takes values in the rational numbers. Consequently, $\modular_G(\Aut(G))$ cannot be an open subset of $(0,\infty)$ as it is a subset of $\Q$. This shows that $\modular_G$ is not an open map.

$\ref{it:idcomp1} \iff \ref{it:idcomp6}$: Note that $G \times \widehat{G} \cong \R^{2d} \times H \times \widehat{H}$. If $K$ is a compact, open subgroup of $H$, then $K^{\perp}$ is a compact, open subgroup of $\widehat{H}$. Hence $K \times K^{\perp}$ is a compact, open subgroup of $H \times \widehat{H}$. Using the equivalence of $\ref{it:idcomp1}$ and $\ref{it:idcomp5}$ with $G \times \widehat{G}$ in place of $G$, we see that $\modular_{\tfp{G}}$ is open if and only if $2d \geq 1$, which happens if and only if $d \geq 1$. Thus $\ref{it:idcomp1}$ and $\ref{it:idcomp6}$ are equivalent. This finishes the proof.
\end{proof}

The following proposition describes how the volume of a lattice as in \eqref{eq:volume} changes when an automorphism is applied.

\begin{proposition}\label{prop:lattice_aut}
Let $\Lambda$ be a uniform lattice in a locally compact abelian group $G$, and let $\alpha \in \Aut(G)$. Then $\alpha(\Lambda)$ is also a lattice in $G$, and
\[ \vol(\alpha(\Lambda)) = \modular_G(\alpha) \vol(\Lambda) .\]
\end{proposition}

\begin{proof}
If $U$ is any open set in $G$ such that $U \cap \Lambda = \{ 1 \}$, then $\alpha(U)$ is open in $G$ and $\alpha(U) \cap \alpha(\Lambda) = \alpha(U \cap \Lambda) = \{ 1 \}$. Moreover, the map $\tilde{\alpha} \colon G/\Lambda \to G/\alpha(\Lambda)$ given by $x\Lambda \mapsto \alpha(x) \alpha(\Lambda)$ is a topological isomorphism. This shows that $\alpha(\Lambda)$ is a lattice in $G$.

Fix a Haar measure on $G$, equip the lattices $\Lambda$ and $\alpha(\Lambda)$ with the counting measure, and let $\mu_{G/\Lambda}$ and $\mu_{G/\alpha(\Lambda)}$ be measures on $G/\Lambda$ and $G/\alpha(\Lambda)$ respectively, such that Weil's formula \eqref{eq:weil} holds. The pushforward measure of $\mu_{G/\Lambda}$ along $\tilde{\alpha}$ is a Haar measure on $G/\alpha(\Lambda)$. It follows that there exists a constant $K > 0$ such that
\begin{equation}
    \int_{G/\Lambda} g \circ \tilde{\alpha} \dif{\mu_{G/\Lambda}} = K \int_{G/\alpha(\Lambda)} g \dif{\mu_{G/\alpha(\Lambda)}} \label{eq:lattice_aut_K}
\end{equation}
for all $g \in L^1(G/\alpha(\Lambda))$. Letting $f \in L^1(G)$, we have that
\begin{align*}
    \int_G f(x) \dif{\mu_G(x)} &= \modular_G(\alpha) \int_G f(\alpha(x)) \dif{\mu_G(x)} && \text{by \eqref{eq:modular_integral}} \\
    &= \modular_G(\alpha) \sum_{\lambda \in \Lambda} \int_{G/\Lambda} f(\alpha(y)\alpha(\lambda)) \dif{\mu_{G/\Lambda}(y\Lambda)} && \text{by \eqref{eq:weil}} \\
    &= \modular_G(\alpha) K \sum_{\gamma \in \alpha(\Lambda)} \int_{G/\alpha(\Lambda)} f(y' \gamma) \dif{\mu_{G/\alpha(\Lambda)}(y'\Lambda)} && \text{by \eqref{eq:lattice_aut_K}} \\
    &= \modular_G(\alpha) K \int_G f(x) \dif{\mu_G(x)} && \text{by \eqref{eq:weil}} .
\end{align*}
This shows that $K = \modular_G(\alpha)^{-1}$. Setting $g = \mathbbm{1}_{G/\alpha(\Lambda)}$ in \eqref{eq:lattice_aut_K}, we obtain
\begin{align*}
\mu_{G/\Lambda}(G/\Lambda) &= \int_{G/\Lambda} \mathbbm{1}_{G/\Lambda} \dif{\mu_{G/\Lambda}} \\
&= \int_{G/\Lambda} \mathbbm{1}_{G/\alpha(\Lambda)} \circ \tilde{\alpha} \dif{\mu_{G/\Lambda}} \\
&= \modular_G(\alpha)^{-1} \int_{G/\alpha(\Lambda)} \mathbbm{1}_{G/\alpha(\Lambda)} \dif{\mu_{G/\alpha(\Lambda)}} \\
&= \modular_G(\alpha)^{-1} \mu_{G/\alpha(\Lambda)}(G/\alpha(\Lambda)) .
\end{align*}
By definition of the volume of a lattice, we then have that
\[ \vol(\alpha(\Lambda)) = \modular_G(\alpha)\vol(\Lambda) .\]
\end{proof}

\section{Time-frequency analysis}\label{sec:tfa}

As usual, let $G$ be a (second-countable) locally compact abelian group, and let $\widehat{G}$ be the Pontryagin dual of $G$. The group $\tfp{G}$ is called the \emph{time-frequency plane} associated to $G$.

Given $x \in G$ and $\omega \in \widehat{G}$, we define two unitary linear operators $T_x$ and $M_{\omega}$ on $L^2(G)$ by
\begin{equation}
    (T_x f)(t) = f(x^{-1}t), \ \  (M_{\omega} f)(t) = \omega(t)f(t) \label{eq:tf_shift}
\end{equation}
for $t \in G$, $f \in L^2(G)$. These two operators are called \emph{time shift} by $x$ and \emph{frequency shift} by $\omega$, respectively. They obey the following commutation relation:
\begin{equation} 
    M_{\omega} T_x = \omega(x) \, T_x M_{\omega}, \ \  x \in G, \omega \in \widehat{G} \label{eq:commutation_relation}
\end{equation}
We also set $\pi(x,\omega) = M_{\omega} T_x$ and call it a \emph{time-frequency shift}. From \eqref{eq:commutation_relation}, one calculates that
\begin{equation}
    \pi(x,\omega) \, \pi(y,\tau) = \overline{\tau(x)} \, \pi(xy,\omega\tau)
\end{equation}
for $(x,\omega),(y,\tau) \in \tfp{G}$.

Let $\Delta$ be a lattice in the time-frequency plane $\tfp{G}$, and let $g \in L^2(G)$. The set
\[ \mathcal{G}(g,\Delta) \coloneqq \{ \pi(z) g : z \in \Delta \} \]
is called the \emph{Gabor system} with window $g$ over the lattice $\Delta$. If this set is a \emph{frame} for $L^2(G)$, i.e.\ there exist $A,B > 0$ such that
\begin{equation}
    A \, \| f \|_2^2 \leq \sum_{z \in \Delta} |\langle f ,\pi(z) g \rangle |^2 \leq B \, \| f \|_2^2
\end{equation}
for every $f \in L^2(G)$, then $\mathcal{G}(g,\Delta)$ is called a \emph{Gabor frame}.

Fix a Haar measure $\mu$ on $G$. The Fourier transform is the map $\mathcal{F} \colon L^1(G) \to C_0(\widehat{G})$ given by
\begin{equation}
    \widehat{f}(\omega) = \mathcal{F}(f)(\omega) = \int_G f(t) \, \overline{\omega(t)} \dif{\mu(t)}
\end{equation}
for $f \in L^1(G)$, $\omega \in \widehat{G}$. The \emph{dual measure} $\widehat{\mu}$ on $\widehat{G}$ corresponding to the chosen measure $\mu$ on $G$ is the Haar measure on $\widehat{G}$ appropriately scaled such that the \emph{Plancherel formula} holds \cite[(31.1)]{HeRo63} for all $f \in L^1(G) \cap L^2(G)$:
\begin{equation}
    \int_G |f(t)|^2 \dif{\mu(t)} = \int_{\widehat{G}} |\widehat{f}(\omega)|^2 \dif{\widehat{\mu}(\omega)} \label{eq:plancherel}
\end{equation}
From this formula, one extends $\mathcal{F}$ in the usual fashion to a unitary linear map of Hilbert spaces $L^2(G) \to L^2(\widehat{G})$.

\subsection{The Feichtinger algebra}

Let $g \in L^2(G)$. The \emph{short-time Fourier transform} \cite[p.\ 215]{Gr98} of a function $f \in L^2(G)$ (with respect to $g$) is the function $V_g f \colon \tfp{G} \to \C$ given by
\begin{equation}
    V_g f(x,\omega) = \langle f, \pi(x,\omega)g \rangle = \int_G f(t) \, \overline{ \omega(t) \, g(x^{-1}t)} \dif{t}
\end{equation}
for $(x,\omega) \in \tfp{G}$.

The \emph{Feichtinger algebra} \cite{Fe81-2,Ja18} is the set $S_0(G)$ consisting of those functions $f \in L^2(G)$ for which
\begin{equation}
\int_{\tfp{G}} |V_f f(z)| \dif{z} < \infty . \label{eq:feichtinger_alg_def}
\end{equation}
For fixed nonzero $g \in S_0(G)$, the expression
\begin{equation}
    \| f \|_{S_0(G),g} = \int_{\tfp{G}} |V_gf(z)| \dif{z} \label{eq:feichtinger_equivalent}
\end{equation}
for $f \in S_0(G)$ defines a norm on $S_0(G)$ turning it into a Banach space, and any other choice of $g \in S_0(G)$ yields an equivalent norm. Specifically, we have for all $f,g_1,g_2 \in S_0(G)$ that
\begin{equation}
    c \| f \|_{S_0(G),g_2} \leq \| f \|_{S_0(G),g_1} \leq C \| f \|_{S_0(G),g_2} \label{eq:feichtinger_constants}
\end{equation}
where $c = \| g_1 \|_2^2 \| g_2 \|_{S_0(G),g_1}^{-1}$ and $C = \| g_2 \|_2^{-2} \| g_1 \|_{S_0(G),g_2}$ (\cite[Proposition 4.10]{Ja18}). When the function $g$ in \eqref{eq:feichtinger_equivalent} is not important, we will omit it from the notation and just write $\| \cdot \|_{S_0(G)}$.

Let $\mathcal{F}$ denote the Fourier transform on $G$. If $f \in S_0(G)$, then $\mathcal{F}(f) \in S_0(\widehat{G})$. Specifically, one has the following equality of norms for $f,g \in S_0(G)$:
\begin{equation}
    \| \mathcal{F}(f) \|_{S_0(\widehat{G}),\mathcal{F}(g)} = \| f \|_{S_0(G),g} . \label{eq:feichtinger_fourier}
\end{equation}
If $\alpha \in \Aut(G)$, then the dilation operator $D_{\alpha} \colon L^2(G) \to L^2(G)$ given by $D_{\alpha} f = \modular_G(\alpha)^{1/2} f \circ \alpha$ is unitary. It also implements a Banach space isomorphism $S_0(G) \to S_0(G)$, and one has the following equality of norms:
\begin{equation}
    \| D_{\alpha} f \|_{S_0(G),D_{\alpha}g} = \| f \|_{S_0(G),g} \ \ f,g \in S_0(G). \label{eq:feichtinger_aut}
\end{equation}

The Fourier algebra of $G$ is the space
\[ A(G) = \{ f \in C_0(G) : \text{there exists $g \in L^1(\widehat{G})$ such that $\mathcal{F}^{-1}(g) = f$} \} . \]
If $G$ is discrete then $S_0(G) = \ell^1(G)$ and if $G$ is compact then $S_0(G) = A(G)$, cf.\ \cite[Lemma 4.11]{Ja18}. The Feichtinger algebra has a number of different descriptions, and one that we will make use of is the following.

\begin{proposition}\label{prop:feichtinger_description}
Let $G$ be a locally compact abelian group. Let $\tilde{K}$ be a compact set in $\widehat{G}$ with nonempty interior. Then $S_0(G)$ consists of all functions $f \in A(G)$ that can be written as a sum $f = \sum_{n \in \N} M_{\omega_n} g_n$ where $g_n \in L^1(G)$ with $\supp(\widehat{g_n}) \subseteq \tilde{K}$ for each $n \in \N$, $\omega_n \in \widehat{G}$ for each $n \in \N$ and $\sum_{n \in \N} \| g_n \|_1 < \infty$. Moreover, an equivalent norm on $S_0(G)$ is given by
\[ \| f \|_{S_0(G),\tilde{K}} = \inf_{(g_n)_{n \in \N}} \sum_{n \in \N} \| g_n \|_1 \]
where the infimum is taken over all such representations of $f$. In particular, there exists a constant $C \geq 0$ such that for all $f \in S_0(G)$ with $\supp(\widehat{f}) \subseteq \tilde{K}$ we have that
\begin{equation}
    \| f \|_{S_0(G)} \leq C \| f \|_{1} . \label{eq:cs_eq_norm}
\end{equation}
\end{proposition}
\begin{proof}
The first part of the proposition follows from the description of $S_{0}(G)$ by the set $\mathcal{N}$ in \cite[Proposition 8.1]{Ja18}. To see that \eqref{eq:cs_eq_norm} holds, note that if $f \in S_0(G)$ is such that $\supp(\widehat{f}) \subseteq \tilde{K}$, then one gets a representation of $f$ as $f = \sum_{n \in \N} M_{\omega_n} g_n$ where $g_1 = f$, $g_k = 0$ for $k > 1$ and $\omega_k = 1$ for all $k \in \N$. By the description of the $S_0(G)$-norm, we then have that
\[ \| f \|_{S_0(G),\tilde{K}} \leq \sum_{k \in \N} \| g_k \|_1 = \| f \|_1 .\]
The constant $C \geq 0$ then comes from the fact that the norm above is equivalent to any other norm on the Feichtinger algebra.
\end{proof}

We will make use of the following properties of the Feichtinger algebra. See \cite[Corollary 5.5(ii), Theorem 5.7(i), Lemma 4.15(iii)]{Ja18} for proofs.

\begin{proposition}\label{prop:feichtinger_properties}
Let $G$ be a locally compact abelian group. Then the following hold:
\begin{enumerate}
    \item\label{it:feichtinger_properties1} The map $S_0(G) \times S_0(G) \to S_0(\tfp{G})$ given by
    \[ (f,g) \mapsto V_g f \]
    is continuous.
    \item\label{it:feichtinger_properties2} If $H$ is a closed subgroup of $G$, then restriction
    \[ f \mapsto f|_H \]
    is a well-defined, continuous map $S_0(G) \to S_0(H)$. In particular, if $H$ is discrete, this is a continuous map $S_0(G) \to \ell^1(H)$.
    \item\label{it:feichtinger_properties3} For any $f \in S_0(G)$, the map $\tfp{G} \to S_0(G)$ given by
    \[ (x,\omega) \mapsto \pi(x,\omega) f \]
    is continuous.
\end{enumerate}
\end{proposition}

\subsection{Duality theory for regular Gabor frames}

Let $\mu$ be a Haar measure on $G$. If another Haar measure $c \mu$, $c > 0$, is chosen on $G$, then the dual measure is given by $\widehat{c \mu} = c^{-1} \widehat{\mu}$. It follows that the product measure on $\tfp{G}$ constructed from any Haar measure on $G$ and the corresponding dual measure on $\widehat{G}$ is independent of the choice of Haar measure on $G$. Consequently, any lattice $\Delta$ in $\tfp{G}$ has a canonically defined volume as in \eqref{eq:volume}. The volume $\vol(\Delta)$ gives an obstruction to the existence of Gabor frames over $\Delta$ which is summarized in the following theorem:

\begin{proposition}[Density theorem]\label{prop:density_thm}
Let $\Delta$ be a lattice in $\tfp{G}$, where $G$ is a locally compact abelian group. If $\mathcal{G}(g,\Delta)$ is a Gabor frame, where $g \in L^2(G)$, then
\[ \vol(\Delta) \leq 1 .\]
\end{proposition}

The above version of the density theorem for (not necessarily rectangular) lattices in locally compact abelian groups is due to Jakobsen and Lemvig, see \cite[Theorem 5.6]{JaLe16}. For a historical exposition of the density theorem, see \cite{He07}.

Given a lattice $\Delta$ in $\tfp{G}$, the \emph{dual lattice} of $\Delta$ is defined as
\begin{equation}
    \Delta^{\circ} = \big\{ z \in \tfp{G} : \text{$\pi(z)\,\pi(w) = \pi(w)\,\pi(z)$ for all $w \in \Delta$} \big\} . \label{eq:dual_lattice}
\end{equation}
This is also a lattice in $\tfp{G}$, see \cite[p.\ 265]{Ri88} and \cite[p.\ 234]{JaLe16}. In time-frequency analysis $\Delta^{\circ} $ is often called the \emph{adjoint lattice}, see \cite{FeKo98}. 

The dual lattice is central to the duality theory of regular Gabor frames, which we summarize in the following proposition:

\begin{proposition}\label{prop:duality_theory}
Let $\Delta$ be a lattice in $\tfp{G}$, where $G$ is a locally compact abelian group, and let $g,h \in S_0(G)$. Then the following hold:
\begin{enumerate}
    \item\label{it:duality_theory1} The Gabor frame-type operator $S_{g,h,\Delta} \colon L^2(G) \to L^2(G)$ given by
    \begin{equation}
    S_{g,h,\Delta} f = \sum_{z \in \Delta} \langle f, \pi(z) g \rangle \pi(z) h \label{eq:gabor_frame_operator}
    \end{equation}
    is bounded, and has the following representation, where the sum converges absolutely:
    \begin{equation}
        S_{g,h,\Delta} = \frac{1}{\vol(\Delta)} \sum_{z \in \Delta^{\circ}} \langle \pi(z) h, g \rangle \pi(z) . \label{eq:janssen_rep}
    \end{equation}
    \item\label{it:duality_theory2} The Gabor system $\mathcal{G}(g,\Delta)$ is a frame for $L^2(G)$, if and only if the Gabor frame operator $S_{g,\Delta} \coloneqq S_{g,g,\Delta}$ is invertible on $L^2(G)$, if and only if there exists a function $h \in S_0(G)$ such that
    \begin{equation}
        S_{g,h,\Delta} = I .\label{eq:dual_atom}
    \end{equation}
    Furthermore, $h$ satisfies \eqref{eq:dual_atom} if and only if
    \begin{align}
        \langle g, \pi(z) h \rangle = \vol(\Delta) \delta_{z,0}, && \text{   for all $z \in \Delta^{\circ}$ } . \label{eq:wexler_raz}
    \end{align}
    Such a function $h$ is called a \emph{dual window} for $g$, and the particular choice $h = S_{g,\Delta}^{-1} g$ will always be a dual window called the \emph{canonical dual} of $g$.
\end{enumerate}
\end{proposition}

The boundedness of $S_{g,\Delta}$ for $g \in S_0(G)$ is proved in \cite[Corollary A.5]{JaLe16}. The formula in \eqref{eq:janssen_rep} is known as the \emph{Janssen representation} and is proved in \cite[p.\ 250]{JaLe16}. We mention that the same formula for the Schwartz--Bruhat space can already be derived from Rieffel's work on Heisenberg modules \cite{Ri88}.
The relation in \eqref{eq:wexler_raz} is known as the \emph{Wexler--Raz relation}, and part (ii) of \Cref{prop:duality_theory} is proved in \cite[Theorem 6.1]{JaLe16}.

We end this section with a description of the dual lattice of $\alpha(\Delta)$ where $\alpha \in \Aut(G \times \widehat{G})$ and $\Delta$ is a lattice in $G \times \widehat{G}$. In general, for a homomorphism $\phi \colon G \to H$ of locally compact abelian groups, we denote by $\widehat{\phi} \colon \widehat{H} \to \widehat{G}$ the homomorphism given by $\widehat{\phi}(\omega) = \omega \circ \phi$ for $\omega \in \widehat{H}$. We identity $\widehat{\alpha}$ as an automorphism of $\widehat{G} \times G$. Denote by $J \colon G \times \widehat{G} \to \widehat{G} \times G$ the map $J(x,\omega) = (\omega,x^{-1})$ and set $\alpha^{\circ} = J^{-1}\widehat{\alpha}^{-1}J$. Then similarly as in the real case \cite{FeKa04} one checks that
\begin{equation}
    (\alpha(\Delta))^{\circ} = \alpha^{\circ}(\Delta^{\circ}) . \label{eq:dual_lattice_automorphism}
\end{equation}

\section{Deformation results}\label{sec:deformation}

\subsection{Continuity of dilation}

As stated in the introduction, a continuity of dilation result is at the heart of our approach towards proving a deformation result for Gabor frames over locally compact abelian groups. Throughout this section, as in the previous two sections, $G$ is assumed a (second-countable) locally compact abelian group. We will need the following lemma:

\begin{lemma}\label{lem:uniform_cs}
Suppose $\{ f_k : k \in \N \}$ is a family of continuous functions on $G$ and that $K$ is a compact subset of $G$ such that $\supp(f_k) \subseteq K$ for all $k \in \N$. Then there exists a neighbourhood $W \subseteq \Aut(G)$ of $\id_G$ and a compact subset $K'$ of $G$ such that $\supp(f_k \circ \alpha) \subseteq K'$ for all $\alpha \in W$ and $k \in \N$.
\end{lemma}

\begin{proof}
Let $U$ be any neighbourhood of $1$ in $G$ with compact closure. Set $W = \bra(K,U)$ and $K' = K\overline{U}$. If $x \in K$ and $\alpha \in W$, then $\alpha^{-1}(x)x^{-1} \in U$, so $\alpha^{-1}(x) \in xU \subseteq K'$. This shows that $\alpha^{-1}(K) \subseteq K'$ for $\alpha \in W$. Hence, for $k \in \N$ and $\alpha \in W$ we have that
\[ \supp(f_k \circ \alpha) = \alpha^{-1}(\supp(f_k)) \subseteq \alpha^{-1}(K) \subseteq K' .\]
\end{proof}

A proof of the following proposition can be found in \cite[p.\ 78, Proposition 2]{Br48}, but we include a proof here for the sake of completeness.

\begin{proposition}\label{prop:strong_cont}
For any locally compact abelian group $G$ the map
$L^{1}(G) \times \Aut(G)\to L^{1}(G)$, $ (f,\alpha) \mapsto f\circ\, \alpha$
is continuous.
\end{proposition}

\begin{proof}
First, note that if $f,f_0 \in L^1(G)$ and $\alpha, \alpha_0 \in \Aut(G)$, then
\begin{align*}
    \| f \circ \alpha - f_0 \circ \alpha_0 \|_1 &\leq \| f \circ \alpha - f_0 \circ \alpha \|_1 + \| f_0 \circ \alpha - f_0 \circ \alpha_0 \|_1 \\
    &= \modular_G(\alpha)^{-1} \| f - f_0 \|_1 + \| f_0 \circ \alpha - f_0 \circ \alpha_0 \|_1 .
\end{align*}
Thus, by continuity of the Braconnier modular function, it suffices to show that $\alpha \mapsto f \circ \alpha$ is continuous for all $f \in L^1(G)$. In fact, since
\[ \| f \circ \alpha - f \circ \alpha_0 \|_1 = \modular_G(\alpha_0)^{-1} \| f \circ (\alpha \alpha_0^{-1}) - f \|_1 \]
it suffices to show continuity at $\alpha_0 = \id_G$. So, fix $f \in L^1(G)$. To begin with, suppose that $f$ is compactly supported, say on $K$. Let $\epsilon > 0$. Let $U \subseteq G$ be any neighbourhood of the identity with compact closure, and set $W = \bra(K,U)$. Then if $x \in K$ and $\alpha \in W$, we have $\alpha^{-1}(x)x^{-1} \in U$, so $\alpha^{-1}(x) \in xU \subseteq K \overline{U} \eqqcolon C$. This shows that $\supp(f \circ \alpha) = \alpha^{-1}(K) \subseteq C$ whenever $\alpha \in W$. Consequently, the function $|f \circ \alpha - f|$ is supported on the compact set $C$ for all $\alpha \in W$.

Let $\phi \in C_c(G)$ satisfy $0 \leq \phi \leq 1$ and $\phi|_{C} = 1$ pointwise. Since $f$ is continuous and compactly supported, it is uniformly continuous, so there exists an open neighbourhood of the identity $V$ such that
\begin{equation}
    yz^{-1} \in V \implies |f(y) - f(z)| < \frac{\epsilon}{\mu(\supp(\phi))}
\end{equation}
for all $y,z \in G$. Set $W' = \bra(C,V)$. Then we have that
\begin{equation}
    x \in C \text{  and  } \alpha \in W' \implies |f(\alpha(x)) - f(x)| < \frac{\epsilon}{\mu(\supp(\phi))} \label{eq:uniform_cont} .
\end{equation}
Thus, for $\alpha \in W \cap W'$, we have
\begin{align*}
    \| f \circ \alpha - f \|_1 &= \int_{G} | f(\alpha(x)) - f(x) | \dif{x} \\
    &= \int_G |f(\alpha(x)) - f(x)| \phi(x) \dif{x} && \text{(since $\phi|_C = 1$)} \\
    &\leq \sup_{x \in C}|f(\alpha(x)) - f(x)| \int_G \phi(x) \dif{x} \\
    &\leq \frac{\epsilon}{\mu(\supp(\phi))} \cdot \mu(\supp(\phi)) && \text{(by \eqref{eq:uniform_cont})} \\
    &\leq \epsilon.
\end{align*}
This shows continuity for the case when $f$ is in $C_c(G)$.
By a standard approximation argument this extends to all $f \in L^1(G)$.
\end{proof}

We are now ready to prove the following important result:

\begin{theorem}\label{thm:strong_cont_s0}
For any locally compact abelian group $G$ the map $S_0(G) \times \Aut(G) \to S_0(G)$, $ (f,\alpha) \mapsto f \circ \,\alpha$
is continuous.
\end{theorem}

{ \allowdisplaybreaks
\begin{proof}
We begin as in the proof of \Cref{prop:strong_cont} and reduce the problem to showing continuity in the $\Aut(G)$-variable. Let $f,f_0,g\in S_0(G)$ and $\alpha,\alpha_0 \in \Aut(G)$. Then using \eqref{eq:feichtinger_aut} and the explicit constant in the upper bound of \eqref{eq:feichtinger_constants}, we have that
\begin{align*}
    \| f \circ \alpha - f_0 \circ \alpha_0 \|_{S_0(G),g} &\leq \| f \circ \alpha - f_0 \circ \alpha \|_{S_0(G),g} + \| f_0 \circ \alpha - f_0 \circ \alpha_0 \|_{S_0(G),g} \\
    &= \modular_G(\alpha)^{-1} \| f - f_0 \|_{S_0(G),g \circ \alpha^{-1}} + \| f_0 \circ \alpha - f_0 \circ \alpha_0 \|_{S_0(G),g} \\
    &\leq \modular_G(\alpha)^{-1} \| g \|_2^{-2} \| g \circ \alpha^{-1} \|_{S_0(G),g} \| f - f_0 \|_{S_0(G),g} \\
    & \hspace{3 cm} + \| f_0 \circ \alpha - f_0 \circ \alpha_0 \|_{S_0(G),g}.
\end{align*}
We need to show that the term $\| g \circ \alpha^{-1} \|_{S_0(G),g}$ above is bounded above when $\alpha$ is sufficiently close to $\id_G$. To this end, we can assume that $\widehat{g}$ is compactly supported. By \Cref{lem:uniform_cs} and \Cref{prop:dual_iso}, we can find a compact subset $\tilde{K}$ of $\widehat{G}$ and a symmetric neighbourhood $W \subseteq \Aut(G)$ of $\id_G$ such that $\supp(\widehat{g} \circ \widehat{\alpha}) \subseteq \tilde{K}$ for all $\alpha \in W$. By \Cref{prop:feichtinger_description}, there exists a constant $C \geq 0$ such that $\| \widehat{g} \circ \widehat{\alpha} \|_{S_0(\widehat{G}),\widehat{g}} \leq C \| \widehat{g} \circ \widehat{\alpha} \|_1$ for all $\alpha \in W$. By Fourier invariance of $S_0(G)$ (see \eqref{eq:feichtinger_fourier}), we then get that
\[ \| g \circ \alpha^{-1} \|_{S_0(G),g} = \| \widehat{g} \circ \widehat{\alpha}^{-1} \|_{S_0(\widehat{G}),\widehat{g}} \leq C \| \widehat{g} \circ \widehat{\alpha}^{-1} \|_1 = C \modular_G(\alpha) \| \widehat{g} \|_1 \]
for $\alpha \in W$. Thus it suffices to show continuity of the map $\alpha \mapsto f \circ \alpha$ for fixed $f \in S_0(G)$. In fact, it suffices to show continuity of this map at $\alpha = \id_G$ since
\[ \| f \circ \alpha - f \circ \alpha_0 \|_{S_0(G),g} = \modular_G(\alpha_0)^{-1} \| f \circ (\alpha \alpha_0^{-1}) - f \|_{S_0(G),g \circ \alpha_0^{-1}} .\]
Therefore, let $f \in S_0(G)$, and let $\epsilon > 0$. Let $\tilde{C}$ be a compact set in $\widehat{G}$ with nonempty interior. By \cref{prop:feichtinger_description}, we can write
\[ f = \sum_{k \in \N} M_{\omega_k} g_k \]
where $\omega_k \in \widehat{G}$ for each $k$, $g_k \in L^1(G)$ with $\supp(\widehat{g_k}) \subseteq \tilde{C}$ for every $k$ and $\sum_{k \in \N} \| g_k \|_1 < \infty$. Using this expression for $f$, we have that
\begin{align}
    \| f \circ \alpha - f \|_{S_0(G)} &= \left\| \sum_{k \in \N} (M_{\omega_k} g_k) \circ \alpha - \sum_{k \in \N} M_{\omega_k} g_k \right\|_{S_0(G)} \notag \\
    &\leq \sum_{k \in \N} \| M_{\omega_k \circ \alpha} (g_k \circ \alpha) - M_{\omega_k} g_k \|_{S_0(G)} \notag \\
    &\leq \sum_{k \in \N} \| M_{\omega_k \circ \alpha} (g_k \circ \alpha) - M_{\omega_k \circ \alpha} g_k \|_{S_0(G)} \notag  \\
    & \hspace{1.5cm} + \sum_{k \in \N} \| M_{\omega_k \circ \alpha} g_k - M_{\omega_k} g_k \|_{S_0(G)} \notag \\
    &= \sum_{k \in \N} \| g_k \circ \alpha - g_k \|_{S_0(G)} \label{eq:term1}  \\
    & \hspace{1.5cm} + \sum_{k \in \N} \| M_{\omega_k \circ \alpha} g_k - M_{\omega_k} g_k \|_{S_0(G)}. \label{eq:term2}
\end{align}
We must show that both of the terms \eqref{eq:term1} and \eqref{eq:term2} can be made sufficiently small by choosing $\alpha$ sufficiently close to $\id_G$.

Combining \Cref{lem:uniform_cs} and \Cref{prop:dual_iso} again, we can find a compact set $\tilde{K} \subseteq \widehat{G}$ and a neighbourhood $W_1 \subseteq \Aut(G)$ of $\id_G$ such that $\widehat{g_k} \circ \widehat{\alpha}$ is supported on $\tilde{K}$ for all $k \in \N$ and $\alpha \in W_1$. Again, by \cref{prop:feichtinger_description}, but using $\tilde{K}$, there exists a constant $c > 0$ such that
\begin{equation}
    \| h \|_{S_0(G)} \leq c \| h \|_1 \label{eq:norm_ineq}
\end{equation}
for all $h \in S_0(G)$ with $\supp(\widehat{h}) \subseteq \tilde{K}$.

Pick $k_0 \in \N$ such that
\begin{equation}
    \sum_{k > k_0} \| g_k \|_1 < \frac{\epsilon}{16c} . \label{eq:k_choice}   
\end{equation}
By continuity of the Braconnier modular function, we can find a neighbourhood $W_2 \subseteq \Aut(G)$ of $\id_G$ for which $\modular_G(\alpha) < 2$ when $\alpha \in W_2$. Combining this with \eqref{eq:k_choice}, we have that
\begin{equation}
    \alpha \in W_1 \cap W_2 \implies \sum_{k > k_0} \| g_k \circ \alpha \|_1 = \modular_G(\alpha) \sum_{k > k_0} \| g_k \|_1 < \frac{\epsilon}{8c} . \label{implication_n1}
\end{equation}
By \cref{prop:strong_cont}, we can find a neighbourhood $W_3 \subseteq \Aut(G)$ of $\id_G$ for which the implication
\begin{equation}
    \alpha \in W_3 \implies \sum_{k=1}^{k_0} \| g_k \circ \alpha_n - g_k \|_1 < \frac{\epsilon}{4c}
\end{equation}
holds. Hence, when $\alpha \in W_1 \cap W_2 \cap W_3$, we have that
\begin{align*}
    \sum_{k \in \N} \| g_k \circ \alpha - g_k \|_{S_0(G)} &\leq \sum_{k \in \N} c \| g_k \circ \alpha - g_k \|_1 && \text{(using \eqref{eq:norm_ineq})} \\
    &= c \sum_{k = 1}^{k_0} \| g_k \circ \alpha - g_k \|_1 + c \sum_{k > k_0} \| g_k \circ \alpha - g_k \|_1 \\
    &\leq c \sum_{k=1}^{k_0} \frac{\epsilon}{4c} + c \sum_{k > k_0} (\| g_k \circ \alpha \|_1 + \| g_k \|_1) \\
    &\leq \frac{\epsilon}{4} + \frac{\epsilon}{8} + \frac{\epsilon}{16} < \frac{\epsilon}{2} .
\end{align*}
This gives us an estimate on the term \eqref{eq:term1}.

Now for every $h \in S_0(G)$, the map $\widehat{G} \to S_0(G)$, $\omega \mapsto M_{\omega} h$ is continuous by \Cref{prop:feichtinger_properties} $\ref{it:feichtinger_properties3}$. Thus we obtain a neighbourhood $W_4 \subseteq \Aut(G)$ of $\id_G$ such that
\begin{equation}
    \alpha \in W_4 \implies \sum_{k=1}^{k_0} \|  M_{\omega_k \circ \alpha}g_k- M_{\omega_k}g_k \|_{S_0(G)} < \frac{\epsilon}{4} . \label{eq:implies_n3}
\end{equation}
Moreover, we have that
\begin{align*}
    \sum_{k > k_0} \| M_{\omega_k \circ \alpha}g_k - M_{\omega_k} g_k \|_{S_0(G)} &\leq \sum_{k > k_0} ( \|  M_{\omega_k \circ \alpha}g_k \|_{S_0(G)} + \| M_{\omega_k} g_k \|_{S_0(G)} ) \\
    &= 2 \sum_{k > k_0} \| g_k \|_{S_0(G)} \\
    &\leq 2c \sum_{k > k_0} \| g_k \|_1  && \text{(using \eqref{eq:norm_ineq})} \\
    &< \frac{\epsilon}{8}.  && \text{(using \eqref{eq:k_choice})}
\end{align*}
Combining the above estimate with \eqref{eq:implies_n3}, we obtain the following estimate for \eqref{eq:term2}:
\begin{equation}
    \alpha \in W_4 \implies \sum_{k \in \N} \| M_{\omega_k \circ \alpha} g_k - M_{\omega_k} g_k \|_{S_0(G)} < \frac{\epsilon}{2}. \label{eq:final_impl}
\end{equation}
Thus, combining our estimates of terms \eqref{eq:term1} and \eqref{eq:term2}, we see that when $\alpha \in W_1 \cap W_2 \cap W_3 \cap W_4$, we obtain
\[ \| f \circ \alpha - f \|_{S_0(G)} < \frac{\epsilon}{2} + \frac{\epsilon}{2} = \epsilon .\]
This finishes the proof.
\end{proof}
}

\subsection{Deformations of Gabor frames}

This subsection ends with a proof of the linear deformation result for uniform Gabor frames on locally compact abelian groups. Our approach here is similar to \cite{FeKa04}. The following lemma is a generalization of \cite[Lemma 3.5]{FeKa04}, but in our proof we avoid Wiener amalgam spaces completely, instead applying the continuity of maps between the Feichtinger algebras of $G$, $\widehat{G}$ and $\Delta$.

\begin{lemma}\label{lem:cont_comp}
For any locally compact abelian group $G$ and lattice $\Delta$ in $\tfp{G}$ the map 
\[ S_0(G) \times S_0(G) \times \Aut(\tfp{G}) \to \ell^1(\Delta), \ \ (g,h,\alpha) \mapsto (\langle g, \pi(\alpha(z)) h \rangle )_{z \in \Delta} \]
is continuous.
\end{lemma}

\begin{proof}
We can write the map as a composition of three maps: The first is the map $S_0(G) \times S_0(G) \times \Aut(\tfp{G}) \to S_0(\tfp{G}) \times \Aut(\tfp{G})$ given by
\[ (g,h,\alpha) \mapsto (V_{h}g,\alpha) .\]
This is continuous by \Cref{prop:feichtinger_properties} \ref{it:feichtinger_properties1}. The second is the map $S_0(\tfp{G}) \times \Aut(\tfp{G}) \to S_0(\tfp{G})$ given by
\[ (F,\alpha) \mapsto F \circ \alpha .\]
The continuity of this map is precisely \cref{thm:strong_cont_s0}, with $\tfp{G}$ in place of $G$. The final map is $S_0(\tfp{G}) \to = \ell^1(\Delta)$ given by
\[ F \mapsto F|_{\Delta} \]
which is also known to be continuous by \Cref{prop:feichtinger_properties} \ref{it:feichtinger_properties2}.
\end{proof}

The following lemma is a generalization of \cite[Theorem 3.6(i)]{FeKa04}:

\begin{lemma}\label{lem:cont_at_frame}
The map $S_0(G) \times S_0(G) \times \Aut(\tfp{G}) \to \B(L^2(G))$ given by
\[ (g,h,\alpha) \mapsto S_{g,h,\alpha(\Delta)} \]
is continuous at $(g_0,h_0,\alpha_0)$ whenever $\mathcal{G}(g_0,\alpha_0(\Delta))$ and $\mathcal{G}(h_0,\alpha_0(\Delta))$ are dual frames for $L^2(G)$.
\end{lemma}

\begin{proof}
Let $g_0,h_0 \in S_0(G)$ and $\alpha_0 \in \Aut(G \times \widehat{G})$ be such that $\mathcal{G}(g_0,\alpha_0(\Delta))$ and $\mathcal{G}(h_0,\alpha_0(\Delta))$ are dual frames for $L^2(G)$. Let $\epsilon > 0$. For $g,h \in S_0(G)$ and $\alpha \in \Aut(\tfp{G})$, define $c_{g,h,\alpha} \in \ell^1(\Delta^{\circ})$ by
\begin{align}
    c_{g,h,\alpha}(z) = \langle h, \pi(\alpha(z)) g \rangle , \quad z \in \Delta^{\circ} .
\end{align}
By the Wexler--Raz relations (\Cref{prop:duality_theory} \ref{it:duality_theory2}) and \eqref{eq:dual_lattice_automorphism} we have that
\begin{equation}
    c_{g_0,h_0,\alpha_0^{\circ}}(z) = \vol(\alpha_0(\Delta)) \delta_{z,0} = \modular(\alpha_0)\vol(\Delta) \delta_{z,0}
\end{equation}
for all $z \in \Delta^{\circ}$. For $\alpha \in \Aut(\tfp{G})$, set
\[ T_{\alpha} = \frac{1}{\vol(\alpha(\Delta))} \sum_{z \in \Delta^{\circ}} c_{g_0,h_0,\alpha_0^{\circ}}(z) \pi(\alpha^{\circ}(z)) = \frac{\modular(\alpha_0)}{\modular(\alpha)}I . \]
Using the Janssen representation \Cref{prop:duality_theory} (i) we get the following estimate:
\begin{align}
    \| S_{g,h,\alpha(\Delta)} - T_{\alpha} \| &= \left\| \frac{1}{\vol(\alpha(\Delta))} \sum_{z \in \Delta^{\circ}} (c_{g,h,\alpha^{\circ}}(z) - c_{g_0,h_0,\alpha_0^{\circ}}(z)) \pi(\alpha^{\circ}(z)) \right\| \notag \\
    &\leq \frac{1}{\modular_G(\alpha)\vol(\Delta)} \sum_{z \in \Delta^{\circ}} | c_{g,h,\alpha^{\circ}}(z) - c_{g_0,h_0,\alpha_0^{\circ}}(z) | \notag  \label{eq:lemlim1}
\end{align}
Thus, by \Cref{lem:cont_comp} and the continuity of the Braconnier modular function as well as the continuity of the assignment $\alpha \mapsto \alpha^{\circ}$, we can find a neighbourhood $U \times V \times W \subseteq S_0(G) \times S_0(G) \times \Aut(\tfp{G})$ of $(g_0,h_0,\alpha_0)$ such that $\| S_{g,h,\alpha(\Delta)} - T_{\alpha} \| < \epsilon / 2$ for $(g,h,\alpha) \in U \times V \times W$.

Next, since $\mathcal{G}(g_0,\alpha_0(\Delta))$ and $\mathcal{G}(h_0,\alpha_0(\Delta))$ are dual frames, it follows that $S_{g_0,h_0,\alpha_0(\Delta)} = I$. Thus, again by continuity of the Braconnier modular function, we can find a neighbourhood $W' \subseteq \Aut(\tfp{G})$ of $\alpha_0$ such that
\begin{equation}
    \| T_{\alpha} - S_{g_0,h_0,\alpha_0(\Delta)} \| = \left\| \frac{\modular(\alpha_0)}{\modular(\alpha)} I - I \right\| < \frac{\epsilon}{2} \label{eq:lemlim2}
\end{equation}
when $\alpha \in W'$. Combining our two observations, we have that
\[ \| S_{g,h,\alpha(\Delta)} - S_{g_0,h_0,\alpha_0(\Delta)} \| \leq \| S_{g,h,\alpha(\Delta)} - T_{\alpha} \| + \| T_{\alpha} -  S_{g_0,h_0,\alpha_0(\Delta)} \| < \frac{\epsilon}{2} + \frac{\epsilon}{2} = \epsilon\]
when $(g,h,\alpha) \in U \times V \times (W \cap W')$. This finishes the proof.
\end{proof}

We are now ready to prove one of the main results of the paper. With all the previous lemmas and propositions taken care of, the proof below is exactly the same as the proof of \cite[Theorem 3.8]{FeKa04}. The deep fact that a Gabor frame $\mathcal{G}(g,\Delta)$ with $g \in S_0(G)$ admits a dual frame with window also in $S_0(G)$ was proved for $G= \R^n$ in \cite{GrLe04} and generalized to locally compact abelian groups in \cite{Au21}.

\begin{theorem}\label{thm:pert}
Let $G$ be a locally compact abelian group, and let $\Delta$ be a lattice in $\tfp{G}$. Then the set
\[ \{ (g,\alpha) \in S_0(G) \times \Aut(\tfp{G}) : \text{$\mathcal{G}(g,\alpha(\Delta))$ is a Gabor frame for $L^2(G)$} \} \]
is open in $S_0(G) \times \Aut(\tfp{G})$.
\end{theorem}

\begin{proof}
Let $g_0 \in S_0(G)$. Suppose $\alpha_0 \in \Aut(\tfp{G})$ is such that $\mathcal{G}(g_0,\alpha_0(\Delta))$ is a Gabor frame. By \cite[Theorem B]{Au21} we can find a dual frame $\mathcal{G}(h_0,\alpha_0(\Delta))$ with $h_0 \in S_0(G)$. Then $S_{g_0,h_0,\alpha_0(\Delta)}=I$. There exists an open neighbourhood $V$ of $I$ in $\B(L^2(G))$ containing only invertible elements. By \cref{lem:cont_at_frame}, there exists an open set $U$ of $S_0(G) \times S_0(G) \times \Aut(\tfp{G})$ containing $(g_0,h_0,\alpha_0)$ such that $S_{g,h,\alpha(\Delta)} \in V$ for every $(g,h,\alpha) \in U$. Consequently, the set
\[ U' = \{ (g, \alpha) \in S_0(G) \times \Aut(\tfp{G}) : (g,h_0,\alpha) \in U \} \]
is open in $S_0(G) \times \Aut(\tfp{G})$. But then $S_{g,h_0,\alpha(\Delta)}$ is invertible whenever $(g,\alpha) \in U'$.

Now fix $(g,\alpha) \in U'$ and set $h = S_{g,h_0,\alpha(\Delta)}^{-1} h_0$. Since $S_{g,h_0,\alpha(\Delta)}$ commutes with $\pi(\alpha(z))$ for every $z \in \Delta$, we have that
\begin{align*}
    S_{g,h_0,\alpha(\Delta)} S_{g,h,\alpha(\Delta)}f &= S_{g,h_0,\alpha(\Delta)} \sum_{z \in \Delta} \langle f , \pi(\alpha(z)) h \rangle \pi(\alpha(z)) h \\
    &= \sum_{z \in \Delta} \langle \langle f, \pi(\alpha(z)) h \rangle \pi(\alpha(z)) S_{g,h_0,\alpha(\Delta)} h \\
    &= \sum_{z \in \Delta} \langle \langle f, \pi(\alpha(z)) h \rangle \pi(\alpha(z)) h_0 \\
    &= S_{g,h_0,\alpha(\Delta)} f.
\end{align*}

Hence $S_{g,h_0,\alpha(\Delta)} S_{g,h,\alpha(\Delta)} = S_{g,h_0,\alpha(\Delta)}$. Since $S_{g,h_0,\alpha(\Delta)}$ is invertible, we conclude that $S_{g,h,\alpha(\Delta)} = I$. Hence $\mathcal{G}(g,\alpha(\Delta))$ is a Gabor frame, and the proof is finished.
\end{proof}

\section{The Balian--Low theorem}\label{sec:blt}

In this section, we classify the locally compact abelian groups in which a Balian--Low theorem for the Feichtinger algebra holds.

From what we have done so far, it follows easily that the Balian--Low theorem holds in groups with noncompact identity component:

\begin{theorem}\label{thm:balian_low}
If $G$ is a locally compact abelian group with noncompact identity component, then for any lattice $\Delta$ in $\tfp{G}$ with $\vol(\Delta)=1$ and any function $g\in S_0(G)$ the Gabor system  $\mathcal{G}(g,\Delta)$ cannot be a frame for $L^2(G)$.
\end{theorem}

\begin{proof}
Suppose for a contradiction that $\mathcal{G}(g,\Delta)$ is a frame for $L^2(G)$, with $g \in S_0(G)$. By \cref{thm:pert}, there exists an open set $U \subseteq \Aut(\tfp{G})$ such that $\mathcal{G}(g,\alpha(\Delta))$ is a frame for $L^2(G)$ whenever $\alpha \in U$. Moreover, $I \in U$. By \cref{thm:idcomp_equivalences}, the Braconnier modular function of $\tfp{G}$ is open. Thus, the set $V = \modular_{\tfp{G}}(U)$ is an open neighbourhood of $1$ in $(0,\infty)$, so there exists an $\alpha \in \Aut(G)$ with $\modular_G(\alpha) > 1$ such that $\mathcal{G}(g,\alpha(\Delta))$ is a Gabor frame. But then, using \Cref{prop:lattice_aut}, we get
\[ \vol(\alpha(\Delta)) = \modular_{\tfp{G}}(\alpha)\vol(\Delta) >1  .\]
This violates the density theorem (\cref{prop:density_thm}), hence we have a contradiction.
\end{proof}

Our next goal is to show that if $G$ has compact identity component, then the Balian--Low theorem fails in a strong sense. To prepare, we need some lemmas.

\begin{lemma}\label{lem:onb_prod}
For $i=1,2$, let $G_i$ be a locally compact abelian group, let $\Delta_i$ be a lattice in $\tfp{G_i}$ and let $g_i \in L^2(G)$ be such that $\mathcal{G}(g_i, \Delta_i)$ is an orthonormal basis for $L^2(G_i)$. Set $G = G_1 \times G_2$. Then, upon identifying $\tfp{G}$ with $\tfp{G_1} \times \tfp{G_2}$, we have that $\mathcal{G}(g_1 \otimes g_2, \Delta_1 \times \Delta_2)$ is an orthonormal basis for $L^2(G)$. Moreover, if $g_i \in S_0(G_i)$ for $i=1,2$, then $g_1 \otimes g_2 \in S_0(G)$.
\end{lemma}

\begin{proof}
Suppose that $\mathcal{G}(g_i,\Delta_i)$ is an orthonormal basis for $L^2(G_i)$ for $i=1,2$. Note that
\[ \pi(z_1,z_2)(g_1 \otimes g_2) = \pi(z_1)g_1 \otimes \pi(z_2)g_2 , \quad (z_1,z_2) \in \Delta_1 \times \Delta_2 , \]
hence the Gabor system $\mathcal{G}(g_1 \otimes g_2, \Delta_1 \times \Delta_2)$ consists exactly of the tensor products of elements from the orthonormal bases $\mathcal{G}(g_i,\Delta_i)$, $i=1,2$. From this it follows immediately that $\mathcal{G}(g_1 \otimes g_2, \Delta_1 \times \Delta_2)$ is an orthonormal basis.

Also, if $g_i \in S_0(G_i)$ for $i=1,2$, then the formula
\[ V_{g_1 \otimes g_2}(g_1 \otimes g_2)(z_1,z_2) = V_{g_1}g_1(z_1) V_{g_2}g_2(z_2) \]
shows that $g_1 \otimes g_2 \in S_0(G_1 \times G_2)$.
\end{proof}

\begin{lemma}\label{lem:compact_times_discrete}
The following hold:
\begin{enumerate}
    \item If $D$ is a discrete abelian group, then $\Lambda = D$ is a lattice in $D$, $\delta_0 \in S_0(D)$, and $\mathcal{G}(\delta_0, \Lambda \times \Lambda^{\perp})$ is an orthonormal basis for $L^2(D)$.
    \item If $C$ is a compact abelian group, then $\Lambda = \{ 1 \}$ is a lattice in $C$, $\mathbbm{1}_C \in S_0(C)$, and $\mathcal{G}(\mathbbm{1}_C, \Lambda \times \Lambda^{\perp})$ is an orthonormal basis for $L^2(C)$.
    \item If $G = D \times C$ where $D$ is a discrete abelian group and $C$ is a compact abelian group, then $\Lambda = D \times \{ 1 \}$ is a lattice in $G$, $\delta_0 \otimes \mathbbm{1}_C \in S_0(G)$, and $\mathcal{G}(\delta_0 \otimes \mathbbm{1}_C, \Lambda \times \Lambda^{\perp})$ is an orthonormal basis for $L^2(G)$.
\end{enumerate}
\end{lemma}

\begin{proof}
We begin by proving $(i)$. That $\Lambda$ is a lattice follows from the fact that $D$ is discrete and $D/\Lambda = \{ 1 \}$ is compact. Since $S_0(D) = \ell^1(D)$ for $D$ discrete (\cite[Lemma 4.11]{Ja18}), we have that $\delta_0 \in S_0(D)$. Finally, since $\Lambda^{\perp} = \{ 1\} \subseteq \widehat{D}$, the Gabor system $\mathcal{G}(\delta_0,\Lambda \times \Lambda^{\perp})$ is nothing but the canonical orthonormal basis for $\ell^2(D)$. The proof of $(ii)$ is dual, and $(iii)$ follows by applying $(i)$ and $(ii)$ together with \cref{lem:onb_prod}.
\end{proof}

\begin{lemma}\label{lem:finite_index}
Let $G$ be a locally compact abelian group, let $H$ be a closed subgroup of $G$ of finite index and let $\Lambda$ be a lattice in $H$. Then $\Lambda$ is also a lattice in $G$. Furthermore, suppose $g \in L^2(H)$ is such that $\mathcal{G}(g,\Lambda \times \Lambda_H^{\perp})$ is an orthonormal basis for $L^2(H)$, where $\Lambda_H^{\perp}$ denotes the annihilator of $\Lambda$ in $H$. Then there exists $\tilde{g} \in L^2(G)$ such that $\mathcal{G}(\tilde{g},\Lambda \times \Lambda^{\perp})$ is an orthonormal basis for $L^2(G)$, where $\Lambda^{\perp}$ denotes the annihilator of $\Lambda$ in $G$. Finally, if $g \in S_0(H)$, then $\tilde{g} \in S_0(G)$.
\end{lemma}

\begin{proof}
The first statement follows from \Cref{lem:lattice-constructions}~(i).
Let $\{ y_1, \ldots, y_k \}$ be a set of coset representatives for $H$ in $G$, where $k = [G:H]$. We normalize the Haar measure on $G$ such that Weil's formula \eqref{eq:weil} holds with respect to a fixed Haar measure on $H$ and the counting measure on $G/H$. To ease notation we will omit explicit reference to Haar measures $\mu_G$, $\mu_H$, etc.\ when integrating with respect to them. Define $\tilde{g} \colon G \to \C$ by
\[ \tilde{g}(xy_j) = \frac{1}{\sqrt{k}} g(x) \]
for $x \in H$ and $1 \leq j \leq k$.

Now $\Lambda_H^{\perp} \cong \Lambda^{\perp}/H^{\perp}$. Thus the orthonormality of $\mathcal{G}(g, \Lambda \times( \Lambda^{\perp}/H^{\perp}))$ means that $\langle g, \pi(\lambda,\tau H^{\perp}) g \rangle = 0$ if and only if $\lambda = 1$ and $\tau \in H^{\perp}$. We calculate $\langle \tilde{g}, \pi(\lambda,\tau) \tilde{g} \rangle$ in terms of $\langle g, \pi(\lambda,\tau H^{\perp}) g \rangle$ as follows:
\begin{align*}
    \langle \tilde{g}, \pi(\lambda,\tau) \tilde{g} \rangle_{L^2(G)} &= \sum_{j=1}^k \int_H \tilde{g}(xy_j) \overline{ \tau(xy_j) \tilde{g}(\lambda^{-1}xy_j)} \dif{x} \\
    &= \frac{1}{k} \sum_{j=1}^k \int_H g(x) \overline{\tau(y_j)} \overline{\tau(x) g(\lambda^{-1}x)} \dif{x} \\
    &= \frac{1}{k} \left( \sum_{j=1}^k \overline{ \tau(y_j)} \right) \langle g, \pi(\lambda, \tau H^{\perp}) g \rangle_{L^2(H)} \\
    &=
     \begin{cases}
       \frac{1}{k} \sum_{j=1}^k \overline{\tau(y_j)} &\quad\text{if $\lambda = 1$ and $\tau \in H^{\perp}$} \\
       0 &\quad\text{otherwise.} \\ 
     \end{cases}
\end{align*}
But if $\tau \in H^{\perp}$ then $\sum_{j=1}^k \overline{\tau(y_j)} = k \delta_{\tau,1}$ as we are summing a character $\tau \in H^{\perp} \cong \widehat{G/H}$ over all the elements of the finite group $G/H$. Thus $\mathcal{G}(\tilde{g},\Lambda \times \Lambda^{\perp})$ is orthonormal.

We now show how completeness of $\mathcal{G}(\tilde{g}, \Lambda \times \Lambda^{\perp}/H^{\perp})$ in $L^2(H)$ implies completeness of $\mathcal{G}(\tilde{g},\Lambda \times \Lambda^{\perp})$ in $L^2(G)$. Note that
\begin{align}
    \langle f, \pi(\lambda,\tau) \tilde{g} \rangle_{L^2(G)} &= \sum_{j=1}^k \int_H f(xy_j) \overline{\tau(xy_j) \tilde{g}(\lambda^{-1} xy_j)} \dif{x} \notag \\
    &= \frac{1}{\sqrt{k}} \sum_{j=1}^k \overline{\tau(y_j)} \langle f( \cdot y_j) , \pi(\lambda , \tau H^{\perp}) g \rangle_{L^2(H)} . \label{eq:complete}
\end{align}
Now suppose that $\langle f, \pi(\lambda,\tau) \tilde{g} \rangle_{L^2(G)} = 0$ for all $\lambda \in \Lambda$, $\tau \in \Lambda^{\perp}$. We can write $\tau = \tau_1 \tau_2$ uniquely in the sense that $\tau_1 H^{\perp} \in \Lambda^{\perp}/H^{\perp}$ and $\tau_2 \in H^{\perp}$ are uniquely determined. We then get
\[ \langle f, \pi(\lambda,\tau) \tilde{g} \rangle_{L^2(G)} = \frac{1}{\sqrt{k}} \sum_{j=1}^k \overline{\tau_2(y_j)} \overline{\tau_1(y_j)} \langle f( \cdot y_j) , \pi(\lambda , \tau_1 H^{\perp}) g \rangle_{L^2(H)} = 0 \]
for all $\lambda \in \Lambda^{\perp}$, $\tau_1 H^{\perp} \in \Lambda^{\perp} / H^{\perp}$ and $\tau_2 \in H^{\perp}$. Recognizing the above expression as an inner product in $L^2(G/H)$ and using that $H^{\perp}$ is an orthogonal basis for $L^2(G/H)$, we deduce that
\[ \overline{\tau_1(y_j)} \langle f( \cdot y_j), \pi(\lambda, \tau_1 H^{\perp}) g \rangle = 0\]
for all $1 \leq j \leq k$, $\lambda \in \Lambda^{\perp}$ and $\tau_1 H^{\perp} \in \Lambda^{\perp} / H^{\perp}$. By completeness of $\mathcal{G}(g,\Lambda \times \Lambda^{\perp}/H^{\perp})$, this implies that $f(x y_j) = 0$ for all $x \in H$ and $1 \leq j \leq k$, i.e.\ $f = 0$.

Finally, suppose that $g \in S_0(H)$. For $h \in L^2(G)$ Denote by $V_h^G$ the short-time Fourier transform on $\tfp{G}$ and by $V_h^H$ the short-time Fourier transform on $\tfp{H}$. Since $G/H$ is finite, we have that $H^{\perp} \cong \widehat{G/H}$ is finite (and of the same order) as well, say $H^{\perp} = \{ \tau_1, \ldots, \tau_k \}$. Note that $\tau_j(t') = 1$ for each $1 \leq j \leq k$ and $t' \in H$. Using this and Weil's formula \eqref{eq:weil}, we have
\begin{align*}
    \int_{\tfp{G}} | V_{\tilde{g}}^G \tilde{g} (x,\omega) | \dif{x} \dif{\omega} &= \sum_{i,j=1}^k \int_{\widehat{G}/H^{\perp}} \int_H | V_{\tilde{g}}^G \tilde{g} (x'y_j,\omega'\tau_j) | \dif{x'} \dif{(\omega'H^{\perp})} \\
    &= \sum_{i,j=1}^k \int_{\widehat{G}/H^{\perp}} \int_H \left| \sum_{l=1}^k \int_H \tilde{g}(t' y_l) \overline{ (\omega'\tau_j)(t'y_l) \tilde{g}({x'}^{-1}y_j^{-1}t'y_l)} \dif{t'} \right| \dif{x'} \dif{(\omega'H^{\perp})} \\
    &\leq k^{-1/2} \sum_{i,j,l=1}^k \int_{\widehat{G}/H^{\perp}} \int_H \left| \int_H g(t') \overline{ \omega'(t') \tilde{g}(x'^{-1}t'y_j^{-1}y_l)} \dif{t'} \right| \dif{x'} \dif{(\omega'H^{\perp})}.
\end{align*}
where we used that $\tau_j(t') = 1$ in the last line. Now for each $1 \leq j,l \leq k$, we can write $y_j^{-1}y_l = t_{j,l}y_{r_{j,l}}$ for a uniquely determined $t_{j,l} \in H$ and $1 \leq r_{j,l} \leq k$. Hence $\tilde{g}({x'}^{-1}t'y_j^{-1}y_l) = \tilde{g}({x'}^{-1}t't_{i,j}y_{r_{i,j}}) = g({x'}^{-1}t't_{i,j})$, so
\begin{align*}
    \int_{\tfp{G}} | V_{\tilde{g}}^G \tilde{g} (x,\omega) | \dif{x} \dif{\omega} &\leq \sum_{j,l=1}^k \int_{\widehat{G}/H^{\perp}} \int_{H} \left| \int_H g(t) \overline{\omega'(t') g({x'}^{-1} t' t_{j,l})} \dif{t'} \right| \dif{x'} \dif{(\omega'H^{\perp})} \\
    &= \sum_{j,l=1}^k \int_{\widehat{G}/H^{\perp}} \int_H | V_g^H g(x't_{j,l}^{-1},\omega')| \dif{x'} \dif{(\omega'H^{\perp})} \\
    &\leq k^2 \int_{H \times \widehat{H}} | V_g^H g(x',\omega')| \dif{x'} \dif{\omega'}. \\
\end{align*}
The above expression is finite because of the assumption that $g \in S_0(H)$. Hence $\tilde{g} \in S_0(G)$, which finishes the proof.

\end{proof}

\begin{lemma}\label{lem:finite_subgroup}
Let $G$ be a locally compact abelian group, let $F$ be a finite subgroup of $G$ and let $\Lambda$ be a lattice in $G$ that contains $F$. Denote by $p \colon G \to G/F$ the quotient map. Then $p(\Lambda)$ is also a lattice in $G/F$. Moreover, suppose $g \in L^2(G/F)$ is such that $\mathcal{G}(g,p(\Lambda) \times p(\Lambda)^{\perp})$ is an orthonormal basis for $L^2(G/F)$. Then there exists $\tilde{g} \in L^2(G)$ such that $\mathcal{G}(\tilde{g},\Lambda \times \Lambda^{\perp})$ is an orthonormal basis for $L^2(G)$. Finally, if $g \in S_0(G/F)$, then $\tilde{g} \in S_0(G)$.
\end{lemma}

\begin{proof}
This is dual to \Cref{lem:finite_index}. Since $F$ is a finite subgroup of $G$ and $F \subseteq \Lambda$, $F^{\perp}$ is a finite index subgroup of $\widehat{G}$ with $\Lambda^{\perp} \subseteq F^{\perp}$. It follows from \Cref{lem:lattice-constructions}~(i) that $\Lambda^{\perp}$ is a lattice in $F^{\perp}$, hence $\Lambda$ is a lattice in $G/F$. Moreover, if $g \in L^2(G/F)$ gives an orthonormal basis $\mathcal{G}(g,p(\Lambda) \times p(\Lambda)^{\perp})$ for $L^2(G/F)$, then $\mathcal{F}\mathcal{G}(g,p(\Lambda) \times p(\Lambda)^{\perp}) = \mathcal{G}(\widehat{g}, p(\Lambda)^{\perp} \times p(\Lambda))$ is an orthonormal basis for $L^2(F^{\perp})$. Again by \Cref{lem:finite_index}, one gets a function $\gamma \in L^2(\widehat{G})$ such that $\mathcal{G}(\gamma, \Lambda^{\perp} \times \Lambda)$ is an orthonormal basis for $L^2(\widehat{G})$, and its inverse Fourier transform $\tilde{g} = \mathcal{F}^{-1}(\gamma)$ gives us an orthonormal basis $\mathcal{G}(\tilde{g}, \Lambda \times \Lambda^{\perp})$ for $L^2(G)$. By Fourier invariance of the Feichtinger algebra, we have that $\tilde{g} \in S_0(G)$.
\end{proof}

With these lemmas proved, we are ready to prove the following:

\begin{theorem}\label{thm:blt_failure}
Let $G$ be a locally compact abelian group. If $G$ has a compact identity component, then for any lattice $\Lambda$ in $G$ there exists a function $g \in S_0(G)$ such that $\mathcal{G}(g,\Lambda \times \Lambda^{\perp})$ is an orthonormal basis for $L^2(G)$.
\end{theorem}

\begin{proof}
By \Cref{thm:idcomp_equivalences}, we can find a compact open subgroup $K$ of $G$. Now $K \Lambda$ is an open subgroup of $G$, so the quotient $G/(K\Lambda)$ is discrete. In addition, we have a continuous surjection $G/\Lambda \to G/(K\Lambda)$, so $G/(K\Lambda)$ must be compact as well. It follows that $H = K\Lambda$ has finite index in $G$. Also, the intersection $F = K \cap \Lambda$ is finite. By \Cref{lem:finite_index} and \Cref{lem:finite_subgroup}, we can assume that $G = K \Lambda$ and $K \cap \Lambda = \{ 1 \}$, i.e.\ $G \cong K \times \Lambda$. But then the existence of a $g \in S_0(G)$ such that $\mathcal{G}(g,\Lambda \times \Lambda^{\perp})$ is an orthonormal basis for $L^2(G)$ follows from \Cref{lem:compact_times_discrete}.
\end{proof}

\subsection{The Zak transform}

As mentioned in the introduction, we will obtain a generalization of a theorem due to Kaniuth and Kutyniok \cite{KaKu98}, at least in the second-countable setting. Let $\Lambda$ be a lattice in $G$. We call a function $F \colon \tfp{G} \to \C$ \emph{quasiperiodic} (with respect to $\Lambda$) if
\begin{equation}
F(x\lambda,\omega\tau) = \overline{\omega(\lambda)} F(x,\omega) \label{eq:quasiperiodic}
\end{equation}
for all $x \in G$, $\omega \in \widehat{G}$, $\lambda \in \Lambda$ and $\tau \in \Lambda^{\perp}$.

The \emph{Zak transform} of a function $f \in L^2(G)$ with respect to $\Lambda$ is the function $Z_{G,\Lambda} f$ on $\tfp{G}$ given by
\[ Z_{G,\Lambda}f(x,\omega) = \sum_{\lambda \in \Lambda} f(x\lambda) \omega(\lambda) \]
for $(x,\omega) \in \tfp{G}$. The Zak transform of a function satisfies \eqref{eq:quasiperiodic}.

The following proposition is proved in \cite[Theorem 5.5]{En19}.

\begin{proposition}\label{prop:zak}
For any locally compact abelian group $G$ and any lattice $\Lambda$ in $G$ the following statements are equivalent:
\begin{enumerate}
    \item For every $g \in S_0(G)$, the Gabor system $\mathcal{G}(g,\Lambda \times \Lambda^{\perp})$ is not a frame for $L^2(G)$.
    \item For every $g \in L^2(G)$ such that the Zak transform $Z_{G,\Lambda} g$ is continuous, $Z_{G,\Lambda} g$ must have a zero.
    \item Every continuous, quasiperiodic function $F$ on $\tfp{G}$ has a zero.
\end{enumerate}
\end{proposition}

Combining the above proposition with \Cref{thm:balian_low} and \Cref{thm:blt_failure}, we obtain the following result:

\begin{theorem}\label{thm:zak_zeros}
For any locally compact abelian group $G$ and any lattice $\Lambda$ in $G$ the following statements are equivalent:
\begin{enumerate}
    \item The identity component of $G$ is noncompact.
    \item Every continuous, $\Lambda$-quasiperiodic function on $\tfp{G}$ has a zero. In particular, whenever $Z_{G,\Lambda} g$ is continuous for $g \in L^2(G)$, then $Z_{G,\Lambda} g$ has a zero.
\end{enumerate}
\end{theorem}

\begin{proof}
Note that by \eqref{eq:vol_lattice_perp}, $\Delta = \Lambda \times \Lambda^{\perp}$ is a lattice of volume $1$ in $\tfp{G}$. If $G$ has noncompact identity component, it follows from \cref{thm:balian_low} that there is no $g \in S_0(G)$ such that $\mathcal{G}(g,\Delta)$ is a Gabor frame. By \Cref{prop:zak}, this is equivalent to every continuous, quasiperiodic function having a zero. On the other hand, if $G$ has compact identity component then by \Cref{thm:blt_failure}, there exists $g \in S_0(G)$ such that $\mathcal{G}(g,\Lambda \times \Lambda^{\perp})$ is an orthonormal basis for $L^2(G)$. By \Cref{prop:zak}, this implies that some continuous, quasiperiodic function is nowhere zero.
\end{proof}

This removes the assumption of compact generatedness from the main result of \cite{KaKu98}, although it should be remarked that there is no assumption of second-countability in \cite{KaKu98}.

\section{Gabor analysis on the adeles of a global field}\label{sec:adeles}

In this section, we introduce the higher dimensional $S$-adeles, show the existence of Gabor frames over these groups and apply our main results to this class of locally compact abelian groups.

\subsection{The adeles of a global field}

Let $K$ be a field. An \emph{absolute value} on $K$ is a function $| \,\cdot\, | \colon K \to [0,\infty)$ that satisfies the following axioms:
\begin{enumerate}[label=(\roman*)]
    \item For all $x,y \in K$, we have that
    \begin{equation}
         |x+y| \leq |x| + |y| . \label{eq:triangle}
    \end{equation}
    \item For all $x,y \in K$, we have that
    \begin{equation}
         |xy| = |x| |y| .
    \end{equation}
    \item For all $x \in K$, we have that $|x| = 0$ implies $x=0$.
\end{enumerate}

The function $| \cdot |_0$ on $K$ given by
\begin{equation}
    | x |_0 = \begin{cases}
       0 &\quad \text{if $x = 0$},  \\
       1 &\quad \text{if $x \neq 0$} ,\\ 
     \end{cases}
\end{equation}
is an absolute value called the \emph{trivial absolute value}. Two absolute values $| \cdot |_1$ and $| \cdot |_2$ on $K$ are called \emph{equivalent} if there exists $s \geq 0$ such that $| x |_1 = |x|_2^{s}$ for all $x \in K$. This is an equivalence relation on the set of absolute values on $K$. It is easy to establish that an absolute value equivalent to the trivial absolute value is also trivial. A \emph{place} is an equivalence class of nontrivial absolute values on $K$.

An absolute value $| \cdot |$ on $K$ induces a metric on $K$ via
\begin{equation}
    d(x,y) = |x-y|
\end{equation}
for $x,y \in K$. One can show that two absolute values on $K$ are equivalent if and only if their induced metrics define the same topology on $K$. The metric associated to the trivial absolute value is the discrete metric on $K$.

An absolute value $| \cdot |$ on $K$ is called \emph{non-archimedean} if
\begin{equation}
    |x+y| \leq \max \{ |x|, |y| \} \label{eq:ultrametric}
\end{equation}
for all $x,y \in K$. Note that \eqref{eq:ultrametric} implies \eqref{eq:triangle}. An absolute value is called \emph{archimedean} if it is not non-archimedean. Note that if $| \cdot|_1$ and $|\cdot|_2$ are equivalent absolute values then both or neither of them are  non-archimedean. Hence, it makes sense to call a place on $K$ non-archimedean (or archimedean).

A \emph{global field} is a finite field extension of either $\Q$ (in which case it is called a number field) or the field $\F_q(t)$ of rational functions over the finite field $\F_q$ of $q$ elements (in which case it is called a function field).

On $\Q$, the usual Euclidean absolute value $| \cdot |_{\infty}$ given by
\[ |x|_{\infty} = \begin{cases}
       x &\quad \text{if $x \geq 0$},  \\
       -x &\quad \text{if $x < 0$}, \\ 
     \end{cases} \]
is archimedean. For each prime number $p$, there is a non-archimedean absolute value, namely the $p$-adic absolute value $| \cdot |_p$, given by
\[ |x|_p = p^{-k} \]
whenever $x = p^k(a/b)$ where $a,b \in \Z$ and $p$ does not divide $a$ or $b$. One also sets $|0|_p = 0$.

On $K = \F_q(t)$ for $q = p^r$, there is a non-archimedean absolute value $| \cdot |_{\infty}$ given by
\begin{equation}
    \left| \frac{f}{g} \right|_{\infty} = p^{\deg(f) - \deg(g)}
\end{equation}
for $f/g \in \F_q(t)$. Furthermore, for each irreducible polynomial $P \in \F_q[t]$, there is a non-archimedean absolute value $ |\cdot|_P$ given by
\begin{equation}
    \left| h \right|_P = p^{-k}
\end{equation}
whenever $h = P^k (f/g)$ where neither $f$ nor $g$ is divisible by $P$. One also sets $|0|_P = 0$.

The following result is proved in \cite[Theorem 4.30]{RaVa99}.

\begin{proposition}\label{prop:places_prime_fields}
The following hold:
\begin{enumerate}
    \item The Euclidean absolute value $| \cdot|_{\infty}$ and the $p$-adic absolute values $|\cdot|_p$ for $p$ prime give a complete set of representatives for all the places on $K = \Q$.
    \item The absolute value $| \cdot |_{\infty}$ and the absolute values $| \cdot |_P$ for an irreducible polynomial $P$ give a complete set of representatives for all the places on $K = \F_q(t)$.
\end{enumerate}
\end{proposition}

If $v$ is a place on a global field $K$, then we denote by $K_v$ the completion of $K$ with respect to the metric induced by $v$. This will be a locally compact field, i.e.\ a field with a nondiscrete topology for which all the algebraic operations are continuous.

For a global field $K$, let $F$ denote either $\Q$ or $\F_q(t)$ depending on which field $K$ is an extension of. If $v$ is a place on $K$, then one obtains a place $v|_F$ on $F$ by restriction. This place is then classified by \cref{prop:places_prime_fields}. If $F = \F_q(t)$, then we say that $v$ is \emph{infinite} if the restriction of $v$ to $F$ is equivalent to the place given by the absolute value $|\cdot|_\infty$ on $\F_q(t)$. Otherwise, $v$ is called \emph{finite}. If $F = \Q$, we simply define the infinite places to be the archimedean ones, and the infinite places to be the non-archimedean ones. We write $v \mid \infty$ to indicate that $v$ is an infinite place, and $v \nmid \infty$ if it is finite.

For a finite place $v$ on $K$, we set
\[ \algint_v = \{ x \in K_v : |x|_v \leq 1 \} \]
where $| \cdot |_v$ is any representative of $v$. Note that the definition of $\algint_v$ is independent of the choice of absolute value. Also, by \eqref{eq:ultrametric}, $\algint_v$ is a subring of $K_v$. One also sets
\[ \algint = \bigcap_{v \nmid \infty} \algint_v. \]
Another description of $\algint$ is as the integral closure of $\Q$ in $K$ if $K$ is a finite extension of $\Q$, and the integral closure of $\F_q(t)$ in $K$ if $K$ is a finite extension of $\F_q(t)$. We set
\begin{equation}
    K_{\infty} = \prod_{v \mid \infty} K_v .
\end{equation}
Note that this is a finite product, since a global field has only finitely many infinite places (this is a consequence of e.g.\ \cite[Theorem 4.31]{RaVa99}). If $S$ is a set of finite places of $K$, we define $K_S$ to be the following restricted product (\cite[p.\ 180]{RaVa99}):
\begin{equation}
    K_S = \resprod{v \in S}{K_v}{\algint_v} = \left\{ (x_v)_v \in \prod_{v \in S} K_v : \text{$x_v \in \algint_v$ for all but finitely many $v \in S$} \right\} .
\end{equation}
Finally, we define the \emph{$S$-adeles} associated to $K$ to be the locally compact ring
\begin{equation}
    \sadeles{K}{S} = K_{\infty} \times K_S .
\end{equation}
In particular, the \emph{adeles} associated to $K$ is $\adeles{K} = \sadeles{K}{\Sigma}$ with $\Sigma$ being the set of all places on $K$. Note also that $\sadeles{K}{\emptyset} = K_{\infty}$.

\begin{proposition}\label{prop:adeles_power}
Let $F$ be a global field, and let $E$ be a finite extension of $F$, with degree $d = [E:F]$. Then we have that
\[ \sadeles{K}{S} \cong \sadeles{F}{S}^d \]
as topological groups.
\end{proposition}

In the case where $S$ is the set of all places, a proof of the above is given in \cite[Lemma 5.10]{RaVa99}, and this works verbatim also for general $S$. The above proposition shows that additively speaking, we do not get anything more by considering powers of the adeles of a finite extension of $\Q$ or $\F_p(t)$ rather than just the fields $\Q$ and $\F_p(t)$ themselves. We will therefore focus on these two cases.

We also set
\begin{equation}
    \srationals{K}{S} = \bigcap_{v \notin S} \algint_v = \{ x \in K : \text{$|x|_v \leq 1$ for all $v \notin S$}. \}
\end{equation}
Note that $\srationals{K}{\Sigma} = K$ and $\srationals{K}{\emptyset} = \algint$.

\begin{proposition}\label{prop:is_lattice}
The map $\srationals{K}{S} \to \sadeles{K}{S}$ given by $x \mapsto (x)_{v \in S}$ is an injective ring homomorphism with discrete range in $\sadeles{K}{S}$. Moreover, identifying $\srationals{K}{S}$ as a discrete subgroup of $\sadeles{K}{S}$ with this map, we have that
\[ \sadeles{K}{S} / \srationals{K}{S} .\]
is compact.
In particular, $\srationals{K}{S}$ embeds as a lattice in $\sadeles{K}{S}$.
\end{proposition}

\begin{proof}
This proof is very similar, more or less verbatim, to \cite[Lemma~5.11]{RaVa99}, just consider any nonempty set of places instead of all places. In both the number field and the function field case it is important that the infinite place is included.
\end{proof}

\begin{remark}
It is known that $\adeles{K} / \rationals{K} \cong \widehat{\rationals{K}}$, see e.g., \cite[Proposition~7.15]{RaVa99}. Moreover, a consequence of the classification result that we will prove in the next subsection is that for any number field $K$ and set of places $S$, we have
\[
\sadeles{K}{S} / \srationals{K}{S} \cong \widehat{\srationals{K}{S}}
\]
For $K=\Q$ there is a direct proof of this in \cite[Theorem~3.3]{KOQ14}.
\end{remark}

We also consider higher dimensional variants of the $S$-adeles. Note that
\[ \sadeles{K}{S}^n \cong K_{\infty}^n \times \resprod{v \in S}{K_v^n}{\algint_v^n} \]
for every $n \in \N$. As a consequence of \Cref{prop:is_lattice}, we have that $\srationals{K}{S}^n$ embeds as a lattice in $\sadeles{K}{S}^n$ for every $n \in \N$. We typically write elements of $\sadeles{K}{S}^n$ as $(x_{\infty}, (x_v)_{v \in S})$ where $x_{\infty} \in K_{\infty}^n$ and $x_v \in K_v^n$ for each $v \in S$.

\subsection{Classification of lattices in the $S$-adeles over the rationals}
Suppose $K = \Q$. Then by \Cref{prop:places_prime_fields}, there is exactly one infinite place on $\Q$, and it is represented by the Euclidean absolute value. Consequently, $K_{\infty} = \R$. If $v$ is a finite place, then it is represented by the $p$-adic absolute value for some prime number $p$. It follows that $K_v = \Q_p$. Furthermore, $\algint_v = \Z_p$ and $\algint = \Z$. If $S$ is a set of finite places, we think of $S$ as a set of primes. In that case,
\[ \Q_S = \resprod{p \in S}{\Q_p}{\Z_p} . \]
The ring of $S$-adeles associated to $\Q$ becomes
\[ \sadeles{\Q}{S} = \R \times \Q_S .\]
We also get
\[ \srationals{\Q}{S} = \Z[\tfrac{1}{p} : p \in S] ,\]
the ring extension of $\Z$ by all the rational numbers $1/p$ for $p \in S$.

In the following proposition, we determine the automorphism group of $\sadeles{\Q}{S}^n$.

\begin{proposition}\label{prop:automorphisms}
Let $S$ be a set of prime numbers. Then the map
\[ GL_n(\R) \times \resprod{p\in S}{GL_n(\Q_p)}{GL_n(\Z_p)} \to \Aut(\sadeles{\Q}{S}^n) \]
given by sending $A = (A_{\infty},(A_p)_{p\in S})$ to the automorphism $\alpha$ of $\sadeles{\Q}{S}^n$ given by
\[ \alpha(x_{\infty},(x_p)_{p\in S}) = (A_{\infty} x_{\infty}, (A_p x_p)_{p \in S}) \]
is a topological isomorphism.
\end{proposition}

\begin{proof}
Denote the map in the proposition by $\Phi$. Denote the domain of $\Phi$ by $G$. One easily checks that $\Phi$ is a well-defined injective homomorphism.

Next, we verify that $\Phi$ is surjective. Let $\alpha$ be a topological automorphism of $\sadeles{\Q}{S}^n$. Precomposing with the injections $\R^n \to \sadeles{\Q}{S}^n$, $\Q_p^n \to \sadeles{\Q}{S}^n$ ($p \in S$) and postcomposing with the surjections $\sadeles{\Q}{S}^n \to \R^n$, $\sadeles{\Q}{S}^n \to \Q_p^n$ ($p \in S$), we obtain continuous homomorphisms of the form $\R^n \to \Q_p^n$, $\Q_p^n \to \R^n$ and $\Q_p^n \to \Q_{p'}^n$ for $p,p' \in S$. Now the only continuous homomorphism $\R^n \to \Q_p^n$ is the trivial one, since the identity component of $\R^n$ (namely $\R^n$) must map into the identity component of $\Q_p^n$ (namely $\{0\}$). Similarly, a continuous homomorphism $\Q_p^n \to \R^n$ is also trivial, since all of the compact subgroups $(p^{-k}\Z_p)^n$ whose union is $\Q_p^n$ must have compact image in $\R^n$ (hence $\{0\}$). Finally, we claim that if $p \neq p'$ then the only continuous homomorphism $\Q_p^n \to \Q_{p'}^n$ is the zero map: By passing to the components of such a map, it suffices to show that a continuous group homomorphism $h\colon \Q_p \to \Q_{p'}$ must be trivial. Assume for a contradiction that $h(1)\neq 0$. Then $h({p'}^n)=h(1){p'}^n$ for all $n \in \Z$. Moreover, since $p$ and $p'$ are distinct primes we have that ${p'}^n\in\Z_p$ for all $n \in \Z$, hence $\Z[1/p'] \subseteq \Z_p$. Because of the inclusion $\Z[1/p']h(1) = h(\Z[1/p']) \subseteq h(\Z_p)$ and the density of $\Z[1/p']h(1)$ in $\Q_{p'}$ we conclude that $h(\Z_p)$ is dense in $\Q_{p'}$. However, $h(\Z_p)$ is also compact, a contradiction, and thus $h(1)=0$.

We now have that $\alpha$ is of the form
\[ \alpha(x_{\infty},(x_p)_p) = (A_{\infty}x_{\infty}, (A_p x_p)_p ) \]
for an automorphism $A_{\infty}$ of $\R^n$ and automorphisms $A_p$ for each $p \in S$. In other words, $A_{\infty} \in GL_n(\R)$ while $A_p \in GL_n(\Q_p)$ for each $p \in S$.

Let $e_i$ be the $i$-th standard basis vector in $\Q_p^n$. We have that $(A_p x_p)_p \in \Q_S^n$ for all $(x_p)_p \in \Q_S^n$. In particular, setting $x_p = e_i$ for all $p \in S$, we have $(A_p e_i)_p \in \Q_S^n$. Thus, for each $i$, $1 \leq i \leq n$, we have that $A_p e_i \in \Z_p^n$ for all but finitely many $p \in S$. This implies that all the entries of the matrix $A_p$ are in $\Z_p$ for all but finitely many $p \in S$. The same argument applies to $\alpha^{-1}$ which is given by $\alpha^{-1}(x_{\infty}, (x_p)_p) = (A_{\infty}^{-1} x_{\infty}, (A_p^{-1} x_p)_p )$, so all the entries of both $A_p$ and $A_p^{-1}$ are in $\Z_p$ for all but finitely many $p \in S$. This forces $A_p \in GL_n(\Z_p)$ for all but finitely many $p \in S$.

Finally, we show continuity of $\Phi$. If $K$ is a compact subset of $\sadeles{\Q}{S}^n$ and $U$ is a neighbourhood of the identity of $\sadeles{\Q}{S}^n$, then we can find compact sets $K_\infty \subseteq \R^n$, $K_p \subseteq \Q_p^n$ ($p \in S$), and neighbourhoods of the identity $U_\infty \subseteq \R^n$, $U_p \subseteq \Q_p^n$ with $U_p = \Z_p^n$ for all but finitely many $p$, such that
\begin{align*}
    U' &\coloneqq U_{\infty} \times \prod_{p \in S} U_p \subseteq U \\
    K &\subseteq K_{\infty} \times \prod_{p \in S} K_p \eqqcolon K' .
\end{align*}
Since the sets $\bra(K,U)$ ($K$ compact, $U$ neighbourhood of identity) form a neighbourhood basis for $\Aut(\sadeles{\Q}{S}^n)$, it suffices to show that $\Phi^{-1}(\bra(K',U'))$ is open in $G$. Now
\begin{align*}
    \bra(K',U') = \bra(K_{\infty},U_{\infty}) \times \prod_{p \in S} \bra(K_p,U_p).
\end{align*}
We conclude that $\Phi^{-1}(\bra(K',U'))$ is open in $G$ from the fact that each of the maps $GL_n(\R) \to \Aut(\R^n)$ and $GL_n(\Q_p) \to \Aut(\Q_p^n)$, $p \in S$, are continuous.

Since $\sadeles{\Q}{S}^n$ is separable, it follows from \cite[Corollary 1.11]{PeSu78} that $\sadeles{\Q}{S}^n$ is a separable, completely metrizable group. By an open mapping theorem for topological groups \cite[Lemma 1]{Br71} and the remark afterwards, it follows that $\Phi$ must be open. Hence $\Phi$ is a topological isomorphism, which finishes the proof.
\end{proof}

\begin{remark}
Let $\Z_S$ denote the compact subgroup $\prod_{p\in S} \Z_p$ of $\Q_S$. There is another way to look at $\Q_S$, namely as a completion of $\srationals{\Q}{S}$. Indeed, denote by $\langle S\rangle$ the multiplicative subgroup of $\Q^\times$ generated by $S$ and let $\frac{m}{n}\Z$ be an open set in $\srationals{\Q}{S}$ for every pair $m,n\in \langle S\rangle$. Then $\Q_S$ is the Hausdorff completion of $\srationals{\Q}{S}$ with respect to this topology, i.e., its topology is generated by $\frac{m}{n}\Z_S$ for $m,n\in \langle S\rangle$, and $\Z_S$ is the closure of $\Z$, see \cite{KOQ14} for more details.
\end{remark}

\begin{lemma}\label{lem:dense-image}
Let $\Lambda$ be a lattice in $\R^n\times\Q_S^n$, and let $f_1$ and $f_2$ denote the projections from $\R^n\times\Q_S^n$ onto the first and second coordinate, respectively. Then $f_1$ is injective on $\Lambda$ and $f_2(\Lambda)$ is dense in $\Q_S^n$.
\end{lemma}

\begin{proof}
We start by showing that the restriction of $f_1$ to $\Lambda$ is injective. Suppose that an element of $\Lambda$ is mapped to the identity of $\R^n$ by $f_1$, i.e., there exists some $y \in \Q_S^n$ such that $(0,y) \in \Lambda$. For $m \in \N$ big enough, $y' = m y \in \Z_S^n$. It follows that $V=\{ (0,ky') : k\in\Z \}$, and hence its closure $\overline{V}$, is contained in $\Lambda \cap ( \{ 0 \} \times \Z_S^n)$. Moreover, since $\Lambda$ is discrete and $\{ 0 \} \times \Z_S^n$ is compact, then $\overline{V}$ is finite. But then $\{ k y' : k \in \Z \}$ is a finite set, and so the only possibility is that $y'=0$, i.e., $y=0$. This shows that the restriction of $f_1$ to $\Delta$ is injective.

For the second part, set $H=\overline{f_2(\Lambda)}$ in $\Q_S^n$. The map $f_2$ induces a surjection
\[
(\R^n\times\Q_S^n)/\Lambda \to \Q_S^n/H, \quad (x,y)+\Lambda \mapsto y+H.
\]
Since the former is compact, the latter is also compact, meaning that $H=\Q_S^n$. Indeed, if $L$ is a cocompact subgroup of $\Q_S^n$, then the dual of $\Q_S^n/L$, which is isomorphic to $L^\perp$, is discrete. But then $L^\perp$ is a discrete subgroup of $\widehat{\Q_S^n} \cong \Q_S^n$, so $L^\perp$ is trivial by a similar argument as above. Hence $L=\Q_S^n$. It follows that $f_2(\Lambda)$ is dense in $\Q_S^n$.
\end{proof}

\begin{lemma}\label{lem:final-step}
Suppose $\Lambda$ is a lattice in $\sadeles{\Q}{S}^n$ such that $\Lambda \cap (\R^n \times \Z_S^n) = \Z^n$. Then $\Lambda = \srationals{\Q}{S}^n$.
\end{lemma}

\begin{proof}

First, suppose that $(x,y) \in \Lambda$. Since $y \in \Q_S^n$, we can find a number $r$ expressible as a product of primes from $S$ such that $r y \in \Z_S^n$. But then $(rx,ry) \in \Lambda \cap (\R^n \times \Z_S^n)$, so $rx=ry=l \in \Z^n$. Thus $(x,y) = (l/r, l/r) \in \srationals{\Q}{S}^n$, which shows that $\Lambda \subseteq \srationals{\Q}{S}^n$.

We now prove the other inclusion. Denote by $f_2$ the projection of $\sadeles{\Q}{S}^n = \R^n \times \Q_S^n$ onto $\Q_S^n$, and set $D = f_2(\Lambda)$. Since $\Lambda \cap (\R^n \times \Z_S^n) = \Z^n$, we have that $D \cap \Z_S^n = \Z^n$. Moreover, $D$ is dense in $\Q_S^n$ by \Cref{lem:dense-image}.

We claim that $D = \srationals{\Q}{S}^n$. To see that $\srationals{\Q}{S}^n \subseteq D$, let $y \in \srationals{\Q}{S}^n$. Since $D$ is dense and $\Z_S^n$ is open, there must be some $x\in\Z_S^n$ such that $y+x\in D$. Let $m \in \N$ be such that $my \in \Z^n$. Then $my+mx$ belongs to both of the groups $D$ and $\Z_S^n$, and thus it is in $\Z^n$. It follows that $x\in\Z^n$ and therefore $y \in D-x=D$. We conclude that $\srationals{\Q}{S}^n\subseteq D$. For the converse, assume that $y \in D$. There is a natural number $m$ such that $my \in\Z_S^n$. As above, we can choose this number so that $1/m \in\srationals{\Q}{S}^n$. Since $my$ also belongs to $D$, we have that $my\in\Z^n$, and therefore $y\in\srationals{\Q}{S}^n$. Thus $D \subseteq \srationals{\Q}{S}^n$.
Thus, if $(r,r) \in \srationals{\Q}{S}^n \subseteq \sadeles{\Q}{S}^n$, then identifying $\srationals{\Q}{S}^n$ with its image in $\sadeles{\Q}{S}^n$, we have $r \in D$. But then there exists $x \in \R^n$ such that $(x,r) \in \Lambda$. But as we have already seen, $\Lambda \subseteq \srationals{\Q}{S}^n$, so we must have $x=r$ and thus $(r,r) \in \Lambda$. This shows that $\srationals{\Q}{S}^n \subseteq \Lambda$, which finishes the proof.
\end{proof}

\begin{proposition}\label{prop:lattices}
All lattices $\Lambda$ in the group $\sadeles{\Q}{S}^n$ are of the form
\[  \Lambda = A \srationals{\Q}{S}^n = \{ (A_{\infty} q, (A_p q)_{p \in S} ) : q \in \srationals{\Q}{S}^n \} \]
for an automorphism $A = (A_{\infty}, (A_p)_{p \in S})$ of $\sadeles{\Q}{S}$ given as in \cref{prop:automorphisms}.

\end{proposition}

\begin{proof}

Let $n\geq 1$ and suppose that $\Lambda$ is a lattice in $\R^n\times\Q_S^n$. Remark that $\Lambda$ is a discrete second countable space, so $\Lambda$ is countable. Set $\Lambda'=\Lambda\cap(\R^n\times\Z_S^n)$. Our first goal is to show that $\Lambda'$ is isomorphic to $\Z^n$. Consider again the projections $f_1$ and $f_2$ from \Cref{lem:dense-image} and $f_1'$ given by
\[
f_1' \colon \R^n\times\Z_S^n \to \R^n, \quad (x,y)\mapsto x.
\]
Since $\Lambda' \subseteq \Lambda$, \Cref{lem:dense-image} implies that the restriction of $f_1'$ to $\Lambda'$ is injective.

Using the second isomorphism theorem, we get isomorphisms
\[
(\R^n\times\Z_S^n)/\Lambda' \cong ((\R^n\times\Z_S^n)+\Lambda)/\Lambda = (\R^n\times\Q_S^n) /\Lambda,
\]
where the latter group is compact by assumption. Indeed, to see that $f_2(\Lambda)+\Z_S^n=\Q_S^n$, recall that $f_2(\Lambda)$ is dense in $\Q_S^n$ by \Cref{lem:dense-image}, so for each $x\in\Q_S^n$ we can choose $y\in f_2(\Lambda)$ arbitrarily close to $x$, in particular such that $x-y\in\Z_S^n$.

Furthermore, we note that $\Lambda'$ is discrete in $\R^n\times\Z_S^n$, and thus a lattice. Since $f_1'$ is a coordinate projection, it is open. Moreover, its kernel $\Z_S^n$ is compact. It follows from \Cref{lem:lattice-constructions}~(ii) that $f_1'(\Lambda')$ is a lattice in $\R^n$, and therefore it is isomorphic to $\Z^n$.
By injectivity, $\Lambda'$ is also isomorphic to $\Z^n$.

\medskip

We now wish to define automorphisms $A_\infty$ of $\R^n$ and $A_p$ for $p\in S$ such that the product automorphism $A$ on $\sadeles{\Q}{S}^n$ takes $\Lambda'$ to $\Z^n$.

Clearly, we find $A_\infty$ such that $A_\infty( f_1(\Lambda'))=\Z^n\subseteq\R^n$ and thus
\[
\Lambda\cap(\R^n\times\Z_S^n) = \{ (A_\infty^{-1}k,\varphi(k)) : k\in\Z^n \}.
\]
Recall that $f_1$ is injective, so such a function $\varphi$ exists and is clearly a homomorphism.
Let $p\in S$ and let $h_p$ be the projection of $\Z_S^n$ onto $\Z_p^n$.
Since $f_2(\Lambda)$ is dense in $\Q_S^n$ and $\Z_S^n$ is open in $\Q_S^n$, we must have that $\varphi(\Z^n)=f_2(\Lambda)\cap\Z_S^n$ is dense in $\Z_S^n$. Since $h_p$ is continuous, $h_p(\varphi(\Z^n))$ is dense in $\Z_p^n$.
Set $v_i:=h_p(\varphi(e_i))$, where $e_i$ denotes $(0,\dotsc,0,1,0,\dotsc,0)$.
It follows that $\{ \sum_{i=1}^n c_iv_i : c_i\in\Z[\frac{1}{p}] \}$ is dense in $\Q_p^n$, and hence the span of $\{v_i\}_{i=1}^n$ over $\Q_p$ coincides with $\Q_p^n$.

Therefore,
$\{v_i\}_{i=1}^n$ is linearly independent in $\Q_p^n$,
and there exists an automorphism $A_p$ of $\Q_p^n$ such that $A_p(v_i)=e_i$.

We finally apply \cref{lem:final-step} to the automorphism $A = (A_\infty, (A_p)_{p \in S})$ of $\sadeles{\Q}{S}$.
\end{proof}

As a consequence of \Cref{prop:lattices}, we obtain the following description of the quotient $\sadeles{K}{S}/\srationals{K}{S}$ in the case when $K$ is a number field.

\begin{proposition}
Let $K$ be a number field and $S$ a set of places.
Then
\[
\sadeles{K}{S} / \srationals{K}{S} \cong \widehat{\srationals{K}{S}}.
\]
Moreover, every lattice in $\sadeles{K}{S}$ is of the form $\varphi(\srationals{K}{S})$ for some $\varphi\in\Aut(\sadeles{K}{S})$.
\end{proposition}

\begin{proof}
Every nontrivial character $\gamma\in\srationals{K}{S}^\perp$ induces an isomorphism
\[\widehat{\gamma}\colon\sadeles{K}{S}\cong\widehat{\sadeles{K}{S}}\]
given by $\widehat{\gamma}(x)(y)=\gamma(xy)$, see \cite[Proposition~7.1]{RaVa99}. Clearly, $\widehat{\gamma}$ takes $\srationals{K}{S}$ to a subgroup of $\srationals{K}{S}^\perp$. Let $F_\gamma$ be the inverse image of $\srationals{K}{S}^\perp$ under $\widehat{\gamma}$. Then $[F_\gamma:\srationals{K}{S}^\perp]<\infty$. Indeed, the quotient group
\[
F_\gamma / \srationals{K}{S}
\]
is compact since it is a closed subgroup of $\sadeles{K}{S} / \srationals{K}{S}$ which is compact by \Cref{prop:is_lattice}.
Moreover, $\srationals{K}{S}^\perp$ is isomorphic to the dual group of $\sadeles{K}{S} / \srationals{K}{S}$, so it is discrete, and thus $F_\gamma$ is discrete.
Therefore, the quotient is both compact and discrete, hence it must be finite.

By \Cref{lem:lattice-constructions}~(iii) it follows that $F_\gamma$ is a lattice in $\sadeles{K}{S}$. Since $\sadeles{K}{S}$ is isomorphic to $\sadeles{\Q}{S}^n$ for some natural number $n$ by \Cref{prop:adeles_power}, we can use \Cref{prop:lattices} to find an automorphism of $\sadeles{K}{S}$ that takes $\srationals{K}{S}$ to $F_\gamma$. Hence,
\[
\sadeles{K}{S} / \srationals{K}{S} \cong \sadeles{K}{S} / F_\gamma \cong \widehat{\sadeles{K}{S}} / \srationals{K}{S}^\perp \cong \widehat{\srationals{K}{S}}.
\]
The last statement follows by invoking \Cref{prop:automorphisms}.
\end{proof}

\subsection{Gabor frames on the adeles of a number field}

We begin with an existence result for Gabor frames over certain lattices in $\sadeles{\Q}{S}^n$. We will need the following lemma:

\begin{lemma}\label{lem:globalfield_inner_product}
Let $G$ be a locally compact abelian group containing a compact open subgroup $K$, with Haar measure $\mu$ normalized so that $\mu(K) = 1$. Then $\mathbbm{1}_K$ is an element of $S_0(G)$, and for $x \in G$, $\omega \in \widehat{G}$ we have that
\[ \langle \mathbbm{1}_{K}, \pi(x,\omega) \mathbbm{1}_{K} \rangle = \begin{cases} 1, & \text{if $x \in K$ and $\omega \in K^{\perp}$} \\ 0, & \text{otherwise.} \end{cases}  \]
\end{lemma}

\begin{proof}
We have that
\begin{align}
    \langle \mathbbm{1}_{K}, \pi(x,\omega) \mathbbm{1}_{K} \rangle &= \int_{K} \mathbbm{1}_{K}(t) \overline{ \omega(t) \mathbbm{1}_{K}(t-x)} \dif{t} \nonumber \\
    &= \overline{ \int_{K \cap (K + x)} \omega(t) \dif{t} }. \label{eq:globalfield_innerprod}
\end{align}
If $x \notin K$, then $K \cap (K +x) = \emptyset$, so \eqref{eq:globalfield_innerprod} becomes zero. Suppose therefore that $x \in K$. Then $K \cap (K + x) = K$, so that \eqref{eq:globalfield_innerprod} becomes
\begin{equation}
    \langle \mathbbm{1}_{K}, \pi(x,\omega) \mathbbm{1}_{K} \rangle = \overline{ \int_{K} \omega(t) \dif{t} } . \label{eq:globalfield_innerprod2}
\end{equation}
This is equal to $1$ if and only if $\omega \in K^{\perp}$ and zero otherwise.

Finally,
\begin{equation}
    \int_{\widehat{G}} \int_G |\langle \mathbbm{1}_K, \pi(x,\omega) \mathbbm{1}_K \rangle | \dif{x} \dif{\omega} = \left( \int_G \mathbbm{1}_K(x) \dif{x} \right) \left( \int_{\widehat{G}} \mathbbm{1}_{K^{\perp}}(\omega) \dif{\omega} \right) = \mu(K) \widehat{\mu}(K^{\perp}) = 1
\end{equation}
which shows that $\mathbbm{1}_K \in S_0(G)$.
\end{proof}

\begin{theorem}\label{thm:gf_existence_adeles}
Let $S$ be a set of prime numbers. Let $g,h \in S_0(\R^n)$, and let $A_{\infty} \in GL_{2n}(\R)$. Set $A_p = I$ for every $p \in S$, and let $A = (A_{\infty}, (A_p)_{p \in S})$ be the automorphism of $\sadeles{\Q}{S}^n$ as in \cref{prop:automorphisms}. Set $\tilde{g} = g  \otimes_{p \in S} \mathbbm{1}_{\Z_p^n}$ and $\tilde{h} = h  \otimes_{p \in S} \mathbbm{1}_{\Z_p^n}$. Then $\tilde{g},\tilde{h} \in S_0(\sadeles{\Q}{S}^n)$, and the following are equivalent:
\begin{enumerate}
    \item The Gabor systems $\mathcal{G}(g,A_{\infty}\Z^{2n})$ and $\mathcal{G}(h,A_{\infty}\Z^{2n})$ form dual frames for $L^2(\R^n)$.
    \item The Gabor systems $\mathcal{G}(\tilde{g}, A\srationals{\Q}{S}^{2n})$ and $\mathcal{G}(\tilde{h}, A\srationals{\Q}{S}^{2n})$ form dual frames for $L^2(\sadeles{\Q}{S}^n)$.
\end{enumerate}
\end{theorem}

\begin{proof}
First of all, $K \coloneqq \prod_{p \in S} \Z_p^n$ is a compact open subgroup of $\Q_S^n$, so $\otimes_{p \in S} \mathbbm{1}_{\Z_p^n} = \mathbbm{1}_{K}$ is in $S_0(\Q_S^n)$ by \cref{lem:globalfield_inner_product}. It follows that $\tilde{g} = g \otimes \mathbbm{1}_K \in S_0(\sadeles{\Q}{S}^n)$ since $g \in S_0(\R^n)$ and $\mathbbm{1}_K \in S_0(\Q_S^n)$, and similarly for $\tilde{h}$.

By the Wexler--Raz relations from \Cref{prop:duality_theory} \ref{it:duality_theory2}, we have that $\mathcal{G}(\tilde{g}, A(\srationals{\Q}{S}^{2n}))$ and $\mathcal{G}(\tilde{h}, A(\srationals{\Q}{S}^{2n}))$ form dual frames for $L^2(\sadeles{\Q}{S}^n)$ if and only if
\begin{equation}
    \langle \tilde{g}, \pi(z) \tilde{h} \rangle_{L^2(\sadeles{\Q}{S}^n)} = \vol( A\srationals{\Q}{S}^{2n}) \delta_{z,0} \label{eq:adeles_inner_prod}
\end{equation}
for all $z \in (A\srationals{\Q}{S}^{2n})^{\circ}$. As in \eqref{eq:dual_lattice_automorphism} we can write $(A\srationals{\Q}{S}^{2n})^{\circ} = A^{\circ}\srationals{\Q}{S}^{2n}$, where $A_\infty^{\circ} = J (A_\infty^t)^{-1} J^t$ for $J$ the standard symplectic $2n \times 2n$ matrix and $A_p^{\circ} = I$ for all $p \in S$, cf.\ also \cite[p.\ 2014]{FeKa04}. Using the tensor product form of $\tilde{g}$ and $\tilde{h}$, the inner product in \eqref{eq:adeles_inner_prod} translates into
\begin{align*}
    \langle \tilde{g}, \pi(A^{\circ}q) \tilde{h} \rangle_{L^2(\sadeles{\Q}{S}^n)} &= \langle g, \pi(A_{\infty}^{\circ} q) h \rangle_{L^2(\R^n)} \langle \mathbbm{1}_{K} , \pi(q) \mathbbm{1}_{K} \rangle_{L^2(\Q_S^n)} 
\end{align*}
for all $q \in \srationals{\Q}{S}^{2n}$. By \Cref{lem:globalfield_inner_product}, the last inner product above is equal to $1$ if and only if $q \in \Z_p^{2n}$ for each $p \in S$ and $0$ otherwise. Since $\Z_p^{2n} \cap \srationals{\Q}{S}^{2n} = \Z^{2n}$ for each $p \in S$, we have that
\begin{equation}
    \langle \tilde{g}, \pi(A^{\circ}q) \tilde{h} \rangle_{L^2(\sadeles{\Q}{S}^n)} = \begin{cases} \langle g, \pi(A_{\infty}^{\circ}q) h \rangle_{L^2(\R^n)}, & \text{if $q \in \Z^{2n}$} \\ 0, & \text{otherwise} . \end{cases} \label{eq:innprod_red}
\end{equation}
From the Wexler--Raz condition for the Gabor systems $\mathcal{G}(g,A_{\infty}\Z^{2n})$ and $\mathcal{G}(h,A_{\infty}^{\circ}\Z^{2n})$, we have that they form dual frames if and only if
\[ \langle g, \pi(A_{\infty}^{\circ} q) h \rangle_{L^2(\R^n)} = \vol(A_{\infty}\Z^{2n}) \delta_{q,0}, \quad q \in \Z^{2n} . \]
Observing that $\vol(A \srationals{\Q}{S}^{2n}) = |\det (A_{\infty})| = \vol(A_{\infty} \Z^{2n})$, we see from \eqref{eq:innprod_red} and \eqref{eq:adeles_inner_prod} that the Wexler--Raz conditions for the two pairs of Gabor systems to be dual frames are equivalent. This finishes the proof.
\end{proof}

There are many known examples of Gabor frames for $L^2(\R^n)$. Using \Cref{thm:gf_existence_adeles}, one can obtain examples of Gabor frames for $L^2(\sadeles{\Q}{S}^n)$ for any set of primes $S$. Setting $n=1$ and setting $S$ in \Cref{thm:gf_existence_adeles} to be either the set of all primes or the set containing a single prime $p$, one obtains \cite[Theorem 4.2]{EnJaLu19}.

Let $K$ be any number field, let $S$ be a set of finite places of $K$ and let $n \in \N$. Then $K$ is by definition a finite extension of either $\Q$. By \cref{prop:adeles_power}, we have that $\sadeles{K}{S}^n \cong \sadeles{\Q}{S}^{dn}$, where $d = [K:\Q]$. Consequently, we can apply \Cref{thm:gf_existence_adeles} with $dn$ in place of $n$ to get the existence of a Gabor frame for $\sadeles{K}{S}^n$.

Our pertubation result for Gabor frames over LCA groups (\Cref{thm:pert}) applies to the setting of the $n$-dimensional $S$-adeles over the rationals. We state the results for this group below, where we use the description of the automorphism group of $\Aut(\sadeles{\Q}{S}^n)$ as in \Cref{prop:automorphisms}.

\begin{theorem}\label{thm:pert_adeles}
Let $S$ be a set of primes, and let $n \in \N$. The set
\[ \left\{ (g, A_{\infty}, (A_p)_{p \in S} ) \in S_0(\sadeles{\Q}{S}^n) \times \Aut(\sadeles{\Q}{S}^n) : \text{$\mathcal{G}(g,A \srationals{\Q}{S}^n)$ is a Gabor frame for $L^2(\sadeles{\Q}{S}^n)$} \} \right\} \]
is open in $S_0(\sadeles{\Q}{S}^n) \times \Aut(\sadeles{\Q}{S}^n)$.
\end{theorem}

For the case of the $S$-adeles over a number field, we obtain the following Balian--Low theorem as a consequence of \Cref{thm:balian_low}.

\begin{theorem}\label{thm:blt_adeles1}
Let $K$ be a number field, and let $G = \sadeles{K}{S}^n$. If $\Delta$ is a lattice in $\tfp{G}$ with $\vol(\Delta) = 1$ and $g \in S_0(G)$, then $\mathcal{G}(g,\Delta)$ cannot be a frame for $L^2(G)$.
\end{theorem}

\begin{proof}
By \cref{prop:adeles_power}, $G$ is isomorphic to
\[ \sadeles{\Q}{S}^{dn} \cong \R^{dn} \times \resprod{p \in S}{\Q_p^{dn}}{\Z_p^{dn}} \]
as a topological group, where $d = [K:\Q]$. The identity component of $G$ is
\[ G_0 = \R^{dn} \times \{ 0 \} \]
which is noncompact. The theorem then follows from \Cref{thm:balian_low}.
\end{proof}

\subsection{Gabor frames on the adeles of a function field}

In this subsection, we investigate the setting where $K = \F_q(t)$. The completion of $\F_q(t)$ with respect to the infinite place from \Cref{prop:places_prime_fields} is isomorphic to $\F_q((1/t))$, the ring of formal Laurent series in the indeterminate $1/t$ \cite[p.\ 298 ex.\ 3(b)]{RaVa99}. The ring of integers of $K$ is $\F_q[t]$, the ring of polynomials in the indeterminate $t$.

If we choose $S = \emptyset$, then $\sadeles{K}{\emptyset} = K_{\infty} = \F_q((1/t))$ and $\srationals{K}{\emptyset} = \algint = \F_q[t]$. We can write any formal Laurent series $f(t)$ in the indeterminate $1/t$ as
\[ f(t) = \cdots + \frac{a_{-2}}{t^2} + \frac{a_{-1}}{t} + a_0 + a_1 t + a_2 t^2 + \cdots + a_m t^m \]
for some natural number $m$ and coefficients $a_j \in \F_q$. From this we see that 
additively speaking, the locally compact group $\F_q((1/t))$ is isomorphic to
\[ \F_q((1/t)) \cong \left( \prod_{k=-\infty}^{-1} \F_q \right) \times \left( \bigoplus_{k = 0}^\infty \F_q \right) .\]
Here, the term $\prod_{k=-\infty}^{-1} \F_q$ corresponds to the coefficients of $f(t)$ with negative index, and the term $\oplus_{k=0}^{\infty} \F_q$ corresponds to the coefficients with nonnegative index. Let $D = \oplus_{k=0}^\infty \F_q$. Then $D$ is a direct sum of discrete groups, hence a discrete group. Its dual group is the compact group
\[ \widehat{D} \cong \prod_{k=0}^\infty \widehat{\F_q} \cong \prod_{k=0}^{\infty} \F_q ,\]
which shows that $\F_q((1/t)) \cong D \times \widehat{D}$ as a topological group.

The following proposition shows that contrary to the case of the adeles over a number field, one does not obtain a Balian--Low theorem for the adeles of a function field.

\begin{proposition}\label{thm:blt_adeles2}
Let $K$ be a global function field, let $S$ be a set of finite places and let $n \in \N$. Then the higher dimensional $S$-adeles $\sadeles{K}{S}^n$ has compact identity component. Hence, for any lattice $\Lambda$ in $\sadeles{K}{S}^n$, there exists $g \in S_0(\sadeles{K}{S}^n)$ such that $\mathcal{G}(g,\Lambda \times \Lambda^{\perp})$ is an orthonormal basis for $L^2(\sadeles{K}{S}^n)$.
\end{proposition}

\begin{proof}
Since $K$ is a global function field, it has positive characteristic, say $p$. Any completion $K_v$ of $K$ with respect to a place $v$ will also have characteristic $p$. By \cite[Theorem 4.12]{RaVa99}, a locally compact field of characteristic $p$ is isomorphic to the field of Laurent series in one indeterminate over a finite field, say $K_v \cong \F_{q_v}((1/t))$ with $q_v = p^{r_v}$ for some $f$. As we saw above, the subgroup of Laurent series $\sum_k a_k t^k$ with $a_k = 0$ for $k \geq 0$ sits as a compact open subgroup $H_v$ inside $\F_{q_v}((1/t))$.

Let $v_1, \ldots, v_k$ be the infinite places on $K$. It follows that $H_{\infty} \coloneqq H_{v_1} \times \cdots \times H_{v_k}$ is a compact open subgroup of $K_{\infty} = K_{v_1} \times \cdots \times K_{v_k}$. Furthermore, $H_S \coloneqq \prod_{v \in S} \algint_v$ is a compact open subgroup of $K_S = \resprod{v \in S}{K_v}{\algint_v}$. Thus $H_{\infty}^n \times H_S^n$ is a compact open subgroup of $\sadeles{K}{S}^n$. By \Cref{thm:idcomp_equivalences}, it follows that $\sadeles{K}{S}^n$ has compact identity component. By \Cref{thm:blt_failure}, the conclusion of the proposition follows.
\end{proof}

For e.g.\ $G = \F_q(t)$ and $S = \emptyset$, we can be more specific than in \Cref{thm:blt_adeles2}. As already explained, in this case $G = D \times \widehat{D}$ where $D = \oplus_{k=1}^{\infty} \F_q$, so $G$ is already globally a product of a discrete and a compact group. Therefore, we can appeal to \Cref{lem:compact_times_discrete} to get a specific orthonormal basis over the lattice $\Lambda \times \Lambda^{\perp}$ where $\Lambda = D \times \{ 0 \}$ in $G$. The window $g$ is given by $g(d,c) = \delta_{d,0}$ for $d \in D$, $c \in \widehat{D}$, i.e.\ 
\[ g((a_k)_{k \in \N}, (b_k)_{k \in \N}) = 
     \begin{cases}
       1, &\quad\text{if $a_k = 0$ for all $k \in \N$,} \\
       0, &\quad\text{otherwise,} \\ 
     \end{cases}
\]
for $(a_k)_k \in \oplus_k \F_q$ and $(b_k)_k \in \prod_k \F_q$. Applying \Cref{thm:gf_existence_adeles}, we can now construct Gabor frames over $\sadeles{\F_q(t)}{S}^n$ for sets $S$ of finite places of $\F_q(t)$.

\printbibliography

\end{document}